\documentclass[11pt,a4paper,twoside]{article}
\usepackage{latexsym,amsfonts,amsmath, amsthm, amssymb}
\usepackage{fullpage}
\usepackage{xcolor,bm}
\usepackage[utf8x]{inputenc}

\newtheorem{theorem}{Theorem}
\newtheorem{proposition}[theorem]{Proposition}
\newtheorem{lemma}[theorem]{Lemma}
\newtheorem{corollary}[theorem]{Corollary}

\theoremstyle{definition}

\newtheorem{remark}[theorem]{Remark}

\newcommand{\ls}{\leqslant}
\newcommand{\gs}{\geqslant}

\renewcommand{\Re}{\operatorname{Re}}
\renewcommand{\Im}{\operatorname{Im}}

\begin{document}

\title{\bf On the thin-shell conjecture for the Schatten classes}

\medskip

\author{Jordan Radke and Beatrice-Helen Vritsiou}

\date{}

\maketitle

\begin{abstract}
\small We study the thin-shell conjecture for the Schatten classes. In particular, we establish the conjecture
for the operator norm; we also improve on the best known bound for the Schatten classes,
due to Barthe and Cordero-Erausquin \cite{Barthe-Cordero-2013}, for a few more cases. 
Moreover, we show that a necessary condition for the conjecture to be true for any of the Schatten classes 
is a rather strong negative correlation property: this implies of course that, for the cases for which 
we already have the conjecture  (as for example for the operator norm),
but in fact also for all the cases for which we can get a better estimate 
than the one in \cite{Barthe-Cordero-2013}, this negative correlation property follows.
For the proofs, our starting point is 
techniques that were employed for the Schatten classes 
with regard to other problems in \cite{Konig-Meyer-Pajor-1998} and \cite{Guedon-Paouris-2007}.
\end{abstract}


\section{Introduction}

We work in real, finite-dimensional vector spaces (that can be identified with ${\mathbb R}^m$ for some $m\gs 1$)
which are equipped with a fixed Euclidean structure (or inner product $\langle\cdot, \cdot\rangle$). The Euclidean norm
in all these spaces is denoted by $\|\cdot\|_2$, and the unit Euclidean ball by $B_2^m$.
More generally, we denote by $\|\cdot\|_p$, $p\gs 1$, the $\ell_p$ norm of a vector in ${\mathbb R}^m$,
and by $B_p^m$ the corresponding unit ball.
A convex body $K$ is a compact, convex subset of the space with non-empty interior.
It is called symmetric if $x\in K$ implies $-x\in K$, and it is called centred if 
\begin{equation*}\int_K\langle x, y\rangle \,dx = 0\qquad \hbox{for all} \ y\in {\mathbb R}^m, \end{equation*}
where $dx$ denotes integration with respect to the Lebesgue measure. A convex body $K$ in ${\mathbb R}^m$ is called \emph{isotropic}
if: (i) it has Lebesgue volume 1, (ii) it is centred, and (iii) its covariance matrix is a multiple of the identity, or, in other words,
\begin{equation*} \int_K\langle x, y\rangle^2 \,dx = L_K^2 \|y\|_2^2\qquad \hbox{for all} \ y\in {\mathbb R}^m;\end{equation*}
the number $L_K$ appearing here is an invariant of the affine class of $K$ (that is, of the family of all images of $K$ under an invertible linear, or affine, transformation of ${\mathbb R}^m$), and is called the \emph{isotropic constant} of $K$. 

A famous open question in the asymptotic theory 
of convex bodies is the \emph{hyperplane conjecture} or \emph{isotropic constant conjecture}: 
in one of its many equivalent formulations, it asks whether, for all $m\gs 1$, 
the isotropic constants of all convex bodies in ${\mathbb R}^m$ can be bounded from above by a number which is independent of $m$
(it is already known that, for every convex body $K$ in ${\mathbb R}^m$, $L_K\gs L_{B_2^m} \simeq 1$).  
Bourgain \cite{Bourgain-1991} has shown that,
if $K \subset {\mathbb R}^m$, then 
$L_K\ls C_1\sqrt[4]{m}\log m$,
and Klartag \cite{Klartag-2006} has improved that to $L_K\ls C_2 \sqrt[4]{m}$,
where $C_1, C_2$ are absolute constants, independent of $m$ or $K$ 
(see also \cite{Klartag-EMilman-2012} for an alternative proof of the latter bound).

In \cite{Bobkov-Koldobsky-2003} another quantity for symmetric (or centred) convex bodies $K$ in ${\mathbb R}^m$ was introduced: 
\begin{equation*} \sigma_K^2 = m\,\frac{{\rm Var}_K\bigl(\|{\cal X}\|_2^2\bigr)}{\bigl[{\mathbb E}_K\bigl(\|{\cal X}\|_2^2\bigr)\bigr]^2},\end{equation*}
where ${\cal X}$ is a random vector distributed uniformly in 
$K$. The thin-shell conjecture is the question whether the quantity $\sigma_K$ is bounded from above by 
an absolute constant (namely a constant independent of the dimension $m$)
for all isotropic convex bodies $K$ in ${\mathbb R}^m$. 
The initial interest in this quantity stems from a long-understood principle going back to Sudakov \cite{Sudakov-1978}
and to Diaconis and Freedman \cite{Diaconis-Freedman-1984} that connects strong concentration of mass
with respect to an isotropic probability distribution on ${\mathbb R}^m$ to the existence of many almost-gaussian 
1-dimensional marginals of that distribution. In the case of convex bodies, a sharp quantitative version of this principle
was put forth by Anttila, Ball and Perissinaki \cite{Anttila-Ball-Perissinaki-2003}. In \cite{Klartag-2007a}
Klartag resolved the question of whether a uniform distribution over an isotropic convex body has 
gaussian marginals by employing exactly this principle 
after managing to estimate $\sigma_K^2$ as $o(m)$ for every isotropic convex body $K$ in ${\mathbb R}^m$ (see
also \cite{Fleury-Guedon-Paouris-2007}, and see the introduction of \cite{Klartag-2007a} and the references therein
for more background on the central limit problem for convex bodies). Shortly after, polynomially better (in the dimension) 
estimates for the quantities $\sigma_K$ appeared in \cite{Klartag-2007b} and in \cite{Fleury-2010b}. The best known
general estimate is now due to Guédon and Milman \cite{Guedon-Milman-2011} who showed that, for every
isotropic convex body $K$ in ${\mathbb R}^m$, $\sigma_K^2 \ls Cm^{2/3}$ for some absolute constant $C$.

The thin-shell conjecture is of course interesting in its own right, and it implies more about isotropic convex bodies
than merely the existence of almost gaussian marginals. It is closely related to two other central conjectures in the 
theory of isotropic convex bodies, the abovementioned hyperplane conjecture and the Poincaré 
(or Kannan-Lovász-Simonovits) conjecture: Eldan and Klartag \cite{Eldan-Klartag-2011} have proved
that the worst estimate we have for $\sigma_K$ for an isotropic convex body $K$ in ${\mathbb R}^m$
also serves as an upper bound for the isotropic constants of all convex bodies in ${\mathbb R}^m$, and hence that the thin-shell conjecture
would imply the hyperplane conjecture. On the other hand, 
the thin-shell conjecture is merely a special case of the Poincaré conjecture. 
It has thus been a rather surprising and extremely interesting development that, up to a logarithmic term at least, the thin-shell conjecture
is also equivalent to the Poincaré conjecture; this breakthrough result is due to Eldan \cite{Eldan-2013}.
We refer the reader to the book \cite{BGVV-book} e.g. for more details about these conjectures and the links between them.

Although, as we saw, the general estimates we have for the thin-shell conjecture are far from the conjectured ones, 
there are a few cases of special families of convex bodies for which the conjecture has been resolved optimally. The first 
such case is the family of $\ell_p$ balls: the thin-shell conjecture in this case follows from a subindependence property
established by Ball and Perissinaki \cite{Ball-Perissinaki-1998}, which implies that the cross terms we get 
when we expand the variance of the Euclidean norm are non-positive and hence can be ignored when
we are trying to bound $\sigma_{B_p^m}$ (see also \cite{Anttila-Ball-Perissinaki-2003} for an alternative and more simple 
proof of this subindependence property). 
For the same reason, that is, a subindependence property (see \cite{Pilipczuk-Wojtaszczyk-2008}),
the thin-shell conjecture follows in the broader class of generalised Orlicz balls as well.
Two more important cases are the family of unconditional isotropic convex bodies,
and the simplex, which Klartag has shown (\cite{Klartag-2009} and \cite{Klartag-2013} respectively) 
both satisfy the thin-shell conjecture (see also \cite{Barthe-Cordero-2013} for the case of the simplex).

\medskip

In this paper we study one more special family of convex bodies with respect to the thin-shell conjecture.
Let ${\cal M}_n({\mathbb C})$ denote the space of all $n\times n$ matrices with complex entries (viewed as a real vector space,
that is, so that ${\rm dim}({\cal M}_n({\mathbb C})) = 2n^2$). 
For $T\in {\cal M}_n({\mathbb C})$ and $p\gs 1$, one defines the Schatten $p$-norm of $T$ by
\begin{equation*} \|T\|_{S_p^n} := \|s(T)\|_p = \left(\sum_{i=1}^n s_i(T)^p\right)^{1/p},\end{equation*}
where $s(T) = (s_1(T),\ldots, s_n(T))$ is the non-increasing rearrangement of the singular values of $T$, that is, of the eigenvalues of $(T^\ast T)^{1/2}$.
As usual, by $\|s(T)\|_\infty$ we mean just the maximum of these singular values, namely $s_1(T)$,
and we set $\|T\|_{S_\infty^n} : = \|s(T)\|_\infty$ to be the operator or spectral norm of $T$.
Let $B(S_p^n)$ denote the unit ball of ${\cal M}_n({\mathbb C})$ equipped with the Schatten $p$-norm, and let $E$ be either the whole space ${\cal M}_n({\mathbb C})$,
or the subspace ${\cal M}_n({\mathbb R})$ of all $n\times n$ matrices with real entries, or one of the following classical subspaces: 
of real self-adjoint (or, more simply, symmetric) matrices, of complex self-adjoint (or Hermitian) matrices, of anti-symmetric Hermitian matrices, 
or of complex symmetric matrices. We can also consider the more general space ${\cal M}_n({\mathbb H})$ of all $n\times n$ matrices with 
quaternion entries (viewed again as a real vector space, that is, so that ${\rm dim}({\cal M}_n({\mathbb H})) = 4n^2$), and its subspace of 
Hermitian quaternionic matrices; the Schatten $p$-norm on those spaces is defined in the same way as above.

Let $K_{p,E} \equiv B(S_p^n) \cap E$ stand for the unit ball of the Schatten $p$-norm in one of the above subspaces $E$. Of course 
$K_{p,E}$ is a convex body which, when $E= {\cal M}_n({\mathbb R})$, ${\cal M}_n({\mathbb C})$ or ${\cal M}_n({\mathbb H})$, we also know is isotropic
(more precisely, its homothetic copy $\overline{K}_{p,E}$ that has volume 1 is isotropic),
so it is natural to ask whether it satisfies the thin-shell conjecture. In fact, we can ask the same question for all of the abovementioned Schatten classes $S_p^n\cap E$,
even if we do not know whether $\overline{K}_{p,E}$ is in isotropic position 
(for example, it is known that this may not happen in the subspace of real self-adjoint matrices 
as indicated in the paper \cite{Aubrun-Szarek-2006}).
In all cases it will turn out that bounding 
$\sigma_{K_{p,E}}$ by an absolute constant is equivalent to bounding the variance ${\rm Var}_{K_{p,E}}\bigl(\|{\cal T}\|_2^2\bigr)$ of the Euclidean norm
by an absolute constant times the dimension $d_n := {\rm dim}(E)$ of the space, which for all cases of $E$ that we consider is $\simeq n^2$:
this is due to the fact that, by the methods referred to in the next paragraph, we have for all these Schatten classes that
\begin{equation}\label{eq:KMP-GP-bounds} {\mathbb E}_{K_{p,E}}\bigl(\|{\cal T}\|_2^2\bigr) \simeq d_n \simeq n^2.\end{equation}
The best known bound for ${\rm Var}_{K_{p,E}}\bigl(\|{\cal T}\|_2^2\bigr)$ is due to Barthe and Cordero-Erausquin \cite{Barthe-Cordero-2013},
who generalised Klartag's method in \cite{Klartag-2009}
in a way that allows one to obtain useful estimates on variances of various Lipschitz functions by taking advantage of certain symmetries the 
function and the underlying probability measure (in our case, the uniform measure on ${\overline K}_{p,E}$) possess: 
for example, in the case of the isotropic unit balls of $S_p^n\cap {\cal M}_n({\mathbb F})$, where $p\in [1,\infty]$ 
and ${\mathbb F} = {\mathbb R}$, or ${\mathbb C}$ or ${\mathbb H}$, there are enough symmetries to lead to the bound
\begin{equation}\label{eq:Barthe-Cordero-bound} 
{\rm Var}_{K_p}\bigl(\|{\cal T}\|_2^2\bigr) \ls Cn\cdot d_n \simeq n^3,\end{equation}
or equivalently to the estimate $\sigma_{K_p}^2 = O(n) = O\bigl(\sqrt{{\rm dim}({\cal M}_n({\mathbb F}))}\bigr)$, 
which is an improvement of
what would follow from the general upper bound for the thin-shell conjecture due to Guédon and Milman \cite{Guedon-Milman-2011}.
One of the main results of this paper is an improvement of \eqref{eq:Barthe-Cordero-bound} for $p=\infty$ and for a few more cases.

The unit balls of $S_p^n$
have been studied in the past with respect to other important conjectures 
or questions in Convex Geometry as well: 
\begin{itemize}
\item
in \cite{Konig-Meyer-Pajor-1998} König, Meyer and Pajor established the hyperplane conjecture for them;
\item
in \cite{Guedon-Paouris-2007} Guédon and Paouris studied the behaviour of the Schatten classes with respect to 
concentration of volume, and showed that all but an exponentially small (in the dimension) fraction of the unit balls $K_p$ 
of $S_p^n$ is found in a Euclidean ball of radius twice the average distance of an element in $K_p$ from the origin.
Note that in many ways this question is complementary to the thin-shell conjecture.
\end{itemize}
Not long after \cite{Guedon-Paouris-2007}, Paouris \cite{Paouris-2006} resolved the latter question in the affirmative for all convex bodies in isotropic position 
(as are the unit balls of $S_p^n\cap {\cal M}_n({\mathbb F})$);
however, as should be expected perhaps, the method he used was quite different
from the methods of \cite{Guedon-Paouris-2007} and of \cite{Konig-Meyer-Pajor-1998}, which are very specific to the Schatten classes. 
We use a refinement of the latter methods, a key ingredient of which is the following fundamental fact from Random Matrix Theory 
(see for example \cite{Mehta-2004} or \cite[Proposition 4.1.3]{Anderson-Guionnet-Zeitouni-2010}; see also Section
\ref{sec:Notation} for corresponding results for the other subspaces):

\bigskip

\noindent {\bf Fact 1.} {\it Let $E = {\cal M}_n({\mathbb R})$, or $E= {\cal M}_n({\mathbb C})$, or $ E = {\cal M}_n({\mathbb H})$,
and let $\beta = 1$ or $\beta=2$ or $\beta=4$ respectively. There exists a constant $c_{n,E}$, depending on $n$ and on the subspace $E$, such that, given any function $F:{\mathbb R}^n \to {\mathbb R}^+$ that is symmetric 
$($namely, invariant under permutations of the coordinates of the input$)$ and measurable, we have that
\begin{equation*} \int_E F\bigl(s(T)\bigr)\,dT = c_{n,E} \cdot \int_{{\mathbb R}^n} F\bigl(|x_1|, \ldots, |x_n|\bigr)\cdot f_{2,\beta,\beta-1}(x)\,dx,\end{equation*}
where $s(T)$ is the singular-values-vector of the matrix $T\in E$, and where, for non-negative integers $a,b, c$, we write
\begin{equation*} f_{a,b,c}(x) = \prod_{1\ls i<j\ls n}\big|x_i^a - x_j^a\big|^b\cdot \prod_{1\ls i\ls n}|x_i|^c.\end{equation*}}

\medskip

Making use of Fact 1 in the very specific cases for which we need it, which concern estimation of certain integrals over the balls $K_{p,E}$,
we arrive at the following lemma (used in both \cite{Konig-Meyer-Pajor-1998} and \cite{Guedon-Paouris-2007}).

\smallskip

\begin{lemma}\label{lem:reduction to n-integrals1}
Let $E = {\cal M}_n({\mathbb R})$, or $E= {\cal M}_n({\mathbb C})$, or $ E = {\cal M}_n({\mathbb H})$,
and let $F:{\mathbb R}^n \to {\mathbb R}^+$
be a measurable, symmetric function. Then
\begin{equation}\label{eq1:reduction to n-integrals1} 
\int_{K_{p,E}} F\bigl(s_1(T),\ldots, s_n(T)\bigr)\, dT = c_{n,E}\int_{B_p^n} F\bigl(|x_1|,\ldots,|x_n|\bigr) f_{2,\beta,\beta-1}(x)\,dx,\end{equation}
where $\beta\in \{1,2,4\}$ and $f_{a,b,c}$ are as above. 

If, in addition, $F$ is positively homogeneous of degree $k$ for some real number $k$
$($that is, $F\bigl(rx_1, \ldots, rx_n\bigr) = r^k\cdot F(x_1,\ldots,x_n)$ for all $r>0)$, then $\eqref{eq1:reduction to n-integrals1}$ can also take the following form:
\begin{equation} \label{eq2:reduction to n-integrals1}
\int_{K_{p,E}} F\bigl(s(T)\bigr)\, dT  
= \frac{c_{n,E}}{\Gamma\!\left(1+\frac{d_n+k}{p}\right)}\int_{{\mathbb R}^n} F\bigl(|x_1|,\ldots,|x_n|\bigr) e^{-\|x\|_p^p}f_{2,\beta,\beta-1}(x)\,dx,
\end{equation}
where $d_n = \beta n^2$ is the dimension of $E$.
\end{lemma}

\smallskip

Note that, by using \eqref{eq1:reduction to n-integrals1}, or more generally 
\begin{equation*} \int_{K_{p,E}} F\bigl(s_1(T),\ldots, s_n(T)\bigr)\, dT = c_{n,E}\int_{B_p^n} F(x) f_{a,b,c}(x)\,dx,\end{equation*}
with $F = {\bf 1}_{sB_p^n}$ for different values of $s\in (0,1)$, we readily see that we must have $d_n = {\rm dim}(E) = abn(n-1)/2 + (c+1)n$
(which is in accordance with the values of $a, b,c$ that appear above). Note also that in the sequel we may write $f_{a,b,c,p}$ for the density
\begin{equation*} \exp\bigl(-\|x\|_p^p\bigr)\cdot f_{a,b,c}(x) =\exp\bigl(-\|x\|_p^p\bigr)\cdot \prod_{1\ls i<j\ls n}\big|x_i^a - x_j^a\big|^b\cdot \prod_{1\ls i\ls n}|x_i|^c.\end{equation*}

\medskip

By taking advantage of this type of reduction of estimation of certain integrals over the balls $K_{p,E}$ to 
estimation of integrals over ${\mathbb R}^n$ with respect to the densities $f_{a,b,c,p}$, and also by exploiting
certain symmetry properties of these densities, we manage to establish the thin-shell conjecture for the operator norm
on each one of the three main matrix spaces, ${\cal M}_n({\mathbb R})$, ${\cal M}_n({\mathbb C})$ or ${\cal M}_n({\mathbb H})$.
In fact, the same arguments also work and allow us to obtain the same result for the subspaces of Hermitian matrices,
of anti-symmetric Hermitian matrices, and of complex symmetric matrices as well (even though we do not know whether
the normalised unit ball of $S_p^n$ in these subspaces is in isotropic position); these cases may be of independent interest
although we are not aware right now of any applications of this result to more classical questions concerning these subspaces of random matrices.

A further result we establish is a necessary condition for the thin-shell conjecture to be true on any of the Schatten classes $S_p^n\cap E$,
where $E$ is any of the subspaces mentioned in the previous paragraph:
this necessary condition is a certain type of negative correlation property for the densities $f_{a,b,c,p}$,
as well as for the uniform densities on the balls $K_{p,E}$ under certain conditions.

\smallskip

We move on to giving the exact technical statements of our main results.

\subsection{Outline of the present paper}

Let us briefly describe the main ideas behind our arguments by focusing on the case of $E = {\cal M}_n({\mathbb F})$
where ${\mathbb F} = {\mathbb R}$, or ${\mathbb C}$ or ${\mathbb H}$. Recall that then 
$d_n= {\rm dim}(E) = \beta n^2$, where $\beta =1, 2$ or $4$ respectively, and that for every $p\gs 1$ the ball $\overline{K}_{p,{\cal M}_n({\mathbb F})} = \overline{K}_p$,
normalised so that it has volume 1, is in isotropic position (we will recall the reasons for this in Section \ref{sec:neg-cor-prop-revisited}, where we will need to study
symmetry (or invariance) properties of the balls $K_{p,E}$ in more detail). 
By \eqref{eq:KMP-GP-bounds} we wish to bound the variance ${\rm Var}_{\overline{K}_p}\bigl(\|{\cal T}\|_2^2\bigr)$
by $C\cdot d_n$, where $C$ is an absolute constant.

Given the equality $\|T\|_2 \equiv \|T\|_{HS} = \|T\|_{S_2^n}$ of the Hilbert-Schmidt norm and the Schatten 2-norm of any matrix $T\in {\cal M}_n({\mathbb H})$,
we can expand the quantity ${\rm Var}_{\overline{K}_p}\bigl(\|{\cal T}\|_2^2\bigr)$ in two different ways:
\begin{gather} 
\nonumber
{\rm Var}_{\overline{K}_p}\bigl(\|{\cal T}\|_2^2\bigr) = \int_{\overline{K}_p}\|T\|_2^4\,dT - \left(\int_{\overline{K}_p}\|T\|_2^2\,dT\right)^2
\\ \nonumber
= \sum_{i,j=1}^n\int_{\overline{K}_p}|a_{i,j}(T)|^4\, dT + \sum_{\substack{i,j,k,l=1\\(i,j)\neq (k,l)}}^n \int_{\overline{K}_p} |a_{i,j}(T)|^2|a_{k,l}(T)|^2\,dT - 
\left(\sum_{i,j=1}^n\int_{\overline{K}_p}|a_{i,j}(T)|^2\, dT\right)^2
\\ 
\intertext{where $a_{i,j}(T)$ is the $(i,j)$-th entry of $T$,}
\nonumber
= \int_{\overline{K}_p}\|(s_1(T),\ldots, s_n(T))\|_2^4\,dT - \left(\int_{\overline{K}_p}\|(s_1(T),\ldots, s_n(T))\|_2^2\,dT\right)^2
\\ \label{eq:Var-expanded2}
= \int_{\overline{K}_p}\|(s_1(T),\ldots, s_n(T))\|_4^4\,dT + \sum_{\substack{i,j=1\\i\neq j}}^n \int_{\overline{K}_p} s_i(T)^2s_j(T)^2\,dT - 
\left(\int_{\overline{K}_p}\|(s_1(T),\ldots, s_n(T))\|_2^2\,dT\right)^2.
\end{gather}
Focusing on the second way for now, we build on ideas and techniques from \cite{Konig-Meyer-Pajor-1998} and \cite{Guedon-Paouris-2007}, 
and try, in this context as well, to reduce the estimation of moments of the Euclidean norm (or of other functions) over the balls $K_p$
to estimation of integrals with respect to the densities $f_{a,b,c,p}$ in ${\mathbb R}^n$, which are no longer uniform, or even log-concave, but have strong symmetry properties.
In the cases of ${\cal M}_n({\mathbb F})$ we can do so because of Lemma \ref{lem:reduction to n-integrals1}, 
which we use it with $F(s(T))$ being the Euclidean norm of $T\in K_p$, that is, the $\ell_2$ norm of the singular values,
or the constant function ${\bf 1}$, or the $\ell_4$ norm of the singular values, or the sum of all cross terms in \eqref{eq:Var-expanded2}.
The core of our reduction then is the following 

\begin{proposition}\label{prop:Var Kp and Mp for MnF} 
Let ${\rm Var}_{M_{2,\beta,\beta-1,p}}\bigl(\|x\|_2^2\bigr)$ $($or more briefly ${\rm Var}_{M_p}\bigl(\|x\|_2^2\bigr))$ denote the quantity
\begin{equation*} \frac{M_{2,\beta,\beta-1,p}\bigl(\|x\|_2^4\bigr)}{M_{2,\beta,\beta-1,p}(1)} - 
\left(\frac{M_{2,\beta,\beta-1,p}\bigl(\|x\|_2^2\bigr)}{M_{2,\beta,\beta-1,p}(1)}\right)^2,\end{equation*}
that is, the variance of the Euclidean norm with respect to the density $e^{-\|x\|_p^p} f_{a,b,c}(x)$ where $a=2$ and $b=c+1=\beta$ .
The following relation is true:
\begin{equation*} {\rm Var}_{M_p}\bigl(\|x\|_2^2\bigr) \simeq \max\Bigl\{\sigma_{K_p}^2, \frac{1}{p}\Bigr\}\cdot n^{4/p}.\end{equation*}
\end{proposition}

In fact, this proposition remains valid for all cases of $E$ among the classical subspaces of $n\times n$ matrices that we have mentioned,
where in place of the quantity ${\rm Var}_{M_{2,\beta,\beta-1,p}}\bigl(\|x\|_2^2\bigr)$ we will have ${\rm Var}_{M_{a,b,c,p}}\bigl(\|x\|_2^2\bigr)$
for the non-negative integers $a,b,c$ which make (a variant of) Lemma \ref{lem:reduction to n-integrals1} valid for $E$. Its proof is detailed in Section \ref{sec:reduction-to-Mp}.

Our main task now shifts into studying the expression 
\begin{align}
\nonumber
{\rm Var}_{M_p}\bigl(\|x\|_2^2\bigr) & = \frac{M_p\bigl(\|x\|_2^4\bigr)}{M_p(1)} - \left(\frac{M_p\bigl(\|x\|_2^2\bigr)}{M_p(1)}\right)^2 
\\  \label{eq:terms-in-Var-Mp}
& = \sum_{i=1}^n {\rm Var}_{M_p}(x_i^2)  + \sum_{i\neq j} \left[\frac{M_p\bigl(x_i^2x_j^2\bigr)}{M_p(1)} - \frac{M_p\bigl(x_i^2\bigr)M_p\bigl(x_j^2\bigr)}{(M_p(1))^2}\right],
\end{align}
and the various terms that appear in it.
A key lemma for this purpose, which is central to the methods of
\cite{Konig-Meyer-Pajor-1998} and \cite{Guedon-Paouris-2007}, is the following
integration-by-parts result which allows one to obtain a series of recursive equivalences
that could facilitate the estimation of the integrals at hand or of other similar quantities.

\begin{lemma}\label{lem:int-by-parts}
For every $l = (\epsilon_l, \rho_l) \in \{+1,-1\}^n \times \{\rho\ \hbox{is a permutation of}\ [n]\}$
we consider the following subsets of ${\mathbb R}^n$ that can be written as intersections of $2n-1$ halfspaces:
\begin{equation*} {\cal P}_l :=  \{x: \epsilon_l(i)x_i \gs 0\ \hbox{for all}\ i,\ \hbox{and}\ |x_{\rho_l(1)}|\ls |x_{\rho_l(2)}|\ls \cdots \ls |x_{\rho_l(n)}|\}.\end{equation*}
Let $\xi\gs 0$ and $s> -d_n-\xi$, and let $f:{\mathbb R}^n\setminus\{0\} \to {\mathbb R}$ be an $s$-homogeneous function 
with the property that the product
\begin{equation*} f(x) \cdot f_{a,b,c}(x) = f(x)\cdot  \prod_{1\ls i<j\ls n}\big|\,x_i^a - x_j^a\big|^b\cdot \prod_{1\ls i\ls n}|x_i|^c\end{equation*} 
is $C^1$ in the interior of each of the subsets ${\cal P}_l$, and its partial derivatives can be continuously extended to the border of ${\cal P}_l$
$($except perhaps at the origin$)$.
Then
\begin{multline}\label{eq1:lem:int-by-parts}
(\xi + c +1)M_p\left(f(x)\sum_{i=1}^n|x_i|^\xi\right) = 
\\
pM_p\left(\|x\|_{\xi + p}^{\xi + p}f(x)\right) - M_p\left(\sum_{i=1}^n|x_i|^\xi x_i\frac{\partial f}{\partial x_i}(x)\right) - 
abM_p\left(f(x)\sum_{i=1}^n\sum_{j\neq i}\frac{|x_i|^\xi\,x_i^a}{x_i^a - x_j^a}\right).
\end{multline}
\end{lemma}

\bigskip

Trying to optimise on the way this lemma can be used for our problem, we manage to obtain precise identities (in the place of inequalities 
or equivalences deduced in \cite{Konig-Meyer-Pajor-1998} and \cite{Guedon-Paouris-2007}) which involve the terms appearing in \eqref{eq:terms-in-Var-Mp}
and which allow us to establish the thin-shell conjecture for all $p$ large enough with regard to the dimension. 

\begin{theorem}\label{thm:thin-shell-conj}
Suppose that $p\gs n^t\log n$ for some $t > 0$. Then we have that
\begin{equation*} {\rm Var}_{M_p}\bigl(\|x\|_2^2\bigr) \ls Cn^{4/p}\cdot \max\{n^{2-t}, 1\}\end{equation*}
for some absolute constant $C$. Combined with Proposition $\ref{prop:Var Kp and Mp for MnF}$, this shows that
\begin{equation*}\sigma_{K_p}^2 \ls C^\prime\max\{n^{2-t}, 1\} \end{equation*}
for all $p\gs n^t\log n$, and in particular that $\sigma_{K_p} = O(1)$ for all $p\gs n^2\log n$. 
The latter range includes the case of the operator norm $S_\infty^n$.

We also deduce that $\sigma_{K_p} \gtrsim 1$ for all $p\gtrsim n^2\log n$, and hence, in particular, that $\sigma_{K_\infty} \simeq 1$.
\end{theorem}

All the details and intermediate results leading to the proof of this theorem are presented in Section \ref{sec:main-proofs}. 
In fact, we can obtain the exact same result for the cases of Hermitian matrices, of anti-symmetric Hermitian matrices,
and of complex symmetric matrices (the latter case follows immediately, 
whereas the required adjustments to the arguments for the other two cases are discussed in Section \ref{sec:Special cases});
the common feature of all these subspaces with the spaces ${\cal M}_n({\mathbb F})$, that lets the method go through in all these instances, 
is that, for all of them, we do have a version of Lemma 1 where the density $f_{a,b,c,p}$ that appears is such that $a=2$
(for the case of Hermitian matrices this is an intriguing fact with several other applications, see \cite{Edelman-LaCroix-2015}). 
Unfortunately, to the best of our knowledge, this is not true for the subspaces of real self-adjoint or of Hermitian quaternionic matrices.

The second main result of that section and of the paper is a necessary condition for the thin-shell conjecture to hold true for any of the Schatten classes $S_p^n$, $p\gs 1$.

\begin{theorem}{\rm(\emph{Negative correlation property for the densities $f_{2,b,c,p}$})} \label{thm:neg-cor-prop}
We have 
\begin{align} 
\label{eq1:thm:neg-cor-prop}
\frac{M_p\bigl(x_1^4\bigr)}{M_p(1)} &\gs \Bigl(\dfrac{3}{2} + o(1)\Bigr) \left(\frac{M_p\bigl(x_1^2\bigr)}{M_p(1)}\right)^2, 
\\
\nonumber
\hbox{and hence} \quad n \cdot {\rm Var}_{M_p}\bigl(x_1^2\bigr) &\gs \Bigl(\dfrac{n}{2} + o(n)\Bigr)\left(\frac{M_p\bigl(x_1^2\bigr)}{M_p(1)}\right)^2\simeq n\cdot n^{4/p}.
\end{align}
Therefore, for 
\begin{equation} \label{eq2:thm:neg-cor-prop}
{\rm Var}_{M_p}\bigl(\|x\|_2^2\bigr)  = n\cdot {\rm Var}_{M_p}\bigl(x_1^2\bigr)
 + n(n-1) \left[\frac{M_p\bigl(x_1^2x_2^2\bigr)}{M_p(1)} - \left(\frac{M_p\bigl(x_1^2\bigr)}{M_p(1)}\right)^2\right],
 \end{equation}
 to be bounded by $n^{4/p}$, or even by $o(n\cdot n^{4/p})$, we need the cross terms in \eqref{eq2:thm:neg-cor-prop} to be negative. 
\end{theorem}

\begin{corollary}
Combining Theorems $\ref{thm:thin-shell-conj}$ and $\ref{thm:neg-cor-prop}$, we can conclude 
that the densities $f_{2,b,c,p}$ on ${\mathbb R}^n$ satisfy a negative correlation property for all $p\in [c_0\,\!n\log n, +\infty)$,
where $c_0$ is an absolute constant that can be computed explicitly. The same is true
for the densities ${\bf 1}_{B_p^n}(x)\cdot f_{2,b,c}(x)$ for all $p\in [c_0\,\!n\log n, +\infty]$. 
\end{corollary}

Finally, in Section \ref{sec:neg-cor-prop-revisited} we explore this ``desirable'' negative correlation property further, with particular emphasis
on what it entails for the original uniform densities over the balls $K_p$ in ${\cal M}_n({\mathbb F})$.

\section{Further notation and preliminaries}\label{sec:Notation}

As mentioned in the Introduction, a function $F:{\mathbb R}^m \to {\mathbb R}$ is called \emph{symmetric}
if it is invariant under permutations of the coordinates of the input. 
It is called \emph{positively homogeneous of degree $k$}, where $k\in {\mathbb R}$, or more simply \emph{$k$-homogeneous}, 
if $F\bigl(rx_1, \ldots, rx_n\bigr) = r^k\cdot F(x_1,\ldots,x_n)$ for all $r>0$. 
We write $\overline{A}$ for the homothetic copy of volume 1 of a set $A\subset {\mathbb R}^m$ (as long as this exists).
Lebesgue volume will be denoted by $|A|$, and hence $\overline{A}:=\frac{1}{|A|^{1/m}}A$ whenever $|A|\neq 0$.

Recall that there is a one-to-one and onto correspondence between symmetric convex bodies in ${\mathbb R}^m$ 
and unit balls of norms on ${\mathbb R}^m$ (and thus it makes perfect sense to try to deal with a question about convex bodies
in the specific case of the unit balls of the Schatten classes, or of some other family of norms).  
We refer to \cite{BGVV-book} for background on isotropic convex bodies (or more generally, 
isotropic log-concave measures), as well as for a detailed account on
the three important conjectures about isotropic convex bodies mentioned in the Introduction, 
the hyperplane conjecture, the thin-shell conjecture and the KLS conjecture.

Unless otherwise specified, the letters $c,c^{\prime }, c_1, c_2$ etc. denote absolute positive
constants whose value may change from line to line. Whenever we
write $A\simeq B$ (or $A\lesssim B$) for two quantities $A, B$ related to objects in
${\mathbb R}^m$, we mean that there exist absolute constants
$c_1,c_2>0$, independent of the dimension $m$, such that 
$c_1A\leq B\leq c_2A$ (or $A\ls c_1B$). We will also use the Landau notation:
$A = O(B)$ means the same thing as $A\lesssim B$, whereas $A = o(B)$ means 
that the ratio $A/B$ tends to 0 as the dimension grows to infinity.

\bigskip

We now list the remaining results from Random Matrix Theory which we are going to use,
that are analogous to Fact 1 and its consequence, Lemma \ref{lem:reduction to n-integrals1},
and cover the cases of the other classical subspaces we mentioned in the Introduction: 
the subspaces of real self-adjoint matrices, of Hermitian matrices, and of Hermitian quaternionic matrices, 
or of complex symmetric matrices, or of anti-symmetric Hermitian matrices.
Recall that, if $T = \bigl(a_{i,j}\bigr)_{i,j}$ is an $n\times n$ matrix, then its adjoint matrix is the matrix
$T^\ast = \bigl(\overline{a}_{j,i}\bigr)_{i,j}$ (where $\overline{a} = a_1 -  a_2\cdot i - a_3\cdot j - a_4\cdot k$
for any quaternion $a = a_1 +  a_2\cdot i + a_3\cdot j + a_4\cdot k \in {\mathbb H}$ with $a_1,a_2, a_3, a_4\in {\mathbb R}$).

We first turn to the case of self-adjoint matrices, that is, matrices $T$ which satisfy $T= T^\ast$: remember that such matrices
have $n$ real eigenvalues $e_1(T), e_2(T),\ldots, e_n(T)$, and that their singular values are just the absolute values of
their eigenvalues. The latter fact of course implies that any symmetric function of the singular values of $T$,
such as any of the $S_p^n$ norms, can be thought of as a symmetric function of its eigenvalues as well. 
For such functions we have the following fact (see for example \cite{Mehta-2004} 
or \cite[Proposition 4.1.1]{Anderson-Guionnet-Zeitouni-2010}).

\bigskip

\noindent {\bf Fact 2.} 
{\it Let ${\mathbb F} = {\mathbb R}$, or ${\mathbb C}$ or ${\mathbb H}$, and let $E$ be the subspace of ${\cal M}_n({\mathbb F})$
of matrices $T$ that satisfy $T = T^\ast$.
There exists a constant $c_{n,E}$, depending on $n$ and on the subspace $E$, such that,
given any function $F:{\mathbb R}^n \to {\mathbb R}^+$ that is symmetric and measurable, we have that
\begin{align*} 
\int_E F\bigl(e_1(T), \ldots, e_n(T)\bigr)\,dT &= c_{n,E} \cdot \int_{{\mathbb R}^n} F\bigl(x_1, \ldots, x_n\bigr)\cdot f_{1,\beta,0}(x)\,dx
\\
& = c_{n,E} \cdot \int_{{\mathbb R}^n} F\bigl(x_1, \ldots, x_n\bigr)\cdot  \prod_{1\ls i<j\ls n}\big|x_i - x_j\big|^{\beta}\,dx,
\end{align*}
where $e_1(T), e_2(T),\ldots, e_n(T)$ are the $n$ real eigenvalues of $T$ $($arranged in non-increasing order$)$,
and where $\beta = 1$ if ${\mathbb F} = {\mathbb R}$, $\beta=2$ if ${\mathbb F} = {\mathbb C}$, and $\beta=4$ 
if ${\mathbb F} = {\mathbb H}$.}

\bigskip

%
%
%
%
%

In the case of the complex symmetric matrices we have the following (see \cite[Chapter 3]{Hua-1963}).


\medskip

\noindent {\bf Fact 3.} 
{\it Let $E$ be the subspace of ${\cal M}_n({\mathbb C})$ of complex symmetric matrices, 
namely matrices $T$ with complex entries and the property that $a_{j,i}(T) = a_{i,j}(T)$.
There exists a constant $c_{n,E}$ such that, given any function $F:{\mathbb R}^n \to {\mathbb R}^+$ 
that is symmetric and measurable, we have that
\begin{align*} 
\int_E F\bigl(s_1(T), \ldots, s_n(T)\bigr)\,dT &= c_{n,E} \cdot \int_{{\mathbb R}^n} F\bigl(|x_1|, \ldots, |x_n|\bigr)\cdot f_{2,1,1}(x)\,dx
\\
& = c_{n,E} \cdot \int_{{\mathbb R}^n} F\bigl(|x_1|, \ldots, |x_n|\bigr)\cdot  
\prod_{1\ls i<j\ls n}\big|x_i^2 - x_j^2\big| \cdot \prod_{1\ls i\ls n} |x_i|\,dx,
\end{align*}
where $s_1(T), \ldots, s_n(T)$ are the singular values of $T$ $($arranged in non-increasing order$)$.}

\bigskip

The counterparts of Lemma \ref{lem:reduction to n-integrals1} following from Facts 2 and 3 are obvious,
thus we will not state them.

%

\medskip

Finally, we have the following result for the subspace of ${\cal M}_n({\mathbb C})$ of anti-symmetric Hermitian matrices, 
where antisymmetric means that $T^\ast = -T$ (see \cite[Chapter 13]{Mehta-2004} or \cite[Section 2]{Edelman-LaCroix-2015} 
for an alternative proof).
Recall that the eigenvalues of such a matrix come in pairs, and are of the form $\pm i\theta_1, \ldots, \pm i\theta_s$ if $n = 2s$, where  
$\theta_1\gs\cdots \theta_s\gs 0$ are $s = \lfloor\frac{n}{2}\rfloor$ non-negative real numbers, while, if $n = 2s + 1$, they are of the form 
$\pm i\theta_1, \ldots, \pm i\theta_s, 0$, that is, the matrix $T$ has one additional eigenvalue, which is equal to 0. 
Then the singular values of $T$ are the numbers $\theta_1,\ldots, \theta_s$ with multiplicity two, as well as the number 0
if $n = 2s + 1$.

\bigskip

\noindent {\bf Fact 4.} 
{\it Let $E$ be the subspace of ${\cal M}_n({\mathbb C})$ of anti-symmetric Hermitian matrices equipped 
with the standard Gaussian measure. Then the induced joint probability density of the singular values 
$\theta_1,\ldots, \theta_s$ of the random matrix $T\in E$ is given by
\begin{gather*} 
{\mathbb P}_n\bigl((\theta_1,\ldots,\theta_s)\in A\bigr)
= c_{n,E} \cdot \int_A \prod_{1\ls i<j\ls s}\big|x_i^2 - x_j^2\big|^2\, \exp\bigl(-\|x\|_2^2\bigr)\,dx
\\
\intertext{if $n = 2s$, and by}
{\mathbb P}_n\bigl((\theta_1,\ldots,\theta_s)\in A\bigr) 
= c_{n,E} \cdot \int_A \prod_{1\ls i<j\ls s}\big|x_i^2 - x_j^2\big|^2 \cdot 
\prod_{1\ls i\ls s} |x_i|^2\, \exp\bigl(-\|x\|_2^2\bigr)\,dx
\end{gather*}
if $n = 2s +1$, where $A$ is a $1$-symmetric measurable subset of ${\mathbb R}^s$, and $c_{n,E}$ 
is a constant depending only on $n$.}

\medskip

Fact 4 will allow us in Section \ref{sec:Special cases} to show that the subspaces of anti-symmetric Hermitian matrices,
as well as of Hermitian matrices, satisfy Theorems \ref{thm:thin-shell-conj} and \ref{thm:neg-cor-prop}.

\bigskip

In the sequel we will also need the following result that gives us the order of magnitude of the volume radius of the balls $K_{p,E}$
(far more accurate estimates for the volume of the unit balls of the Schatten classes of real and complex matrices 
have been found by Saint-Raymond \cite{SaintRaymond-1984}, but we won't need those here).

\bigskip

\noindent {\bf Fact 5.} (See \cite[Proposition 3]{Guedon-Paouris-2007})
{\it Let ${\mathbb F} = {\mathbb R}$, or ${\mathbb C}$ or ${\mathbb H}$, and let $E$ be any subspace of ${\cal M}_n({\mathbb F})$
with dimension $d_n \simeq n^2$ $($this includes all the classical subspaces we mentioned in the Introduction$)$.
Then for every $p\gs 1$ we have
\begin{equation*} |K_{p,E}|^{1/d_n} = |B(S_p^n) \cap E|^{1/d_n} 
\simeq d_n^{-\frac{1}{4}-\frac{1}{2p}} \simeq n^{-\frac{1}{2}-\frac{1}{p}}. 
\end{equation*}}

\section{Reduction to integrals over ${\mathbb R}^n$}\label{sec:reduction-to-Mp}

In this section the main purpose is to prove Proposition \ref{prop:Var Kp and Mp for MnF}. We start with detailed estimates about the Gamma function.

\begin{lemma}\label{lem:Gamma estimates}
For every $p\gs 1$, for every dimension $d_n\simeq n^2$ as above, and for every $q\in [2,c_0\:\!d_n]$, the following estimates are true:
\begin{enumerate}
\item[\rm (a)]
$\displaystyle \frac{\Gamma\left(1+\frac{d_n}{p}\right)}{\Gamma\left(1+\frac{d_n+q}{p}\right)} = \left(\frac{d_n+p+q}{p}\right)^{-q/p} 
\left(1 + O\left(\frac{q}{n^2}\right)\right)^q$;
\item[\rm (b)]
$\displaystyle \frac{C_2}{p\,(p+d_n)} \left(\frac{\Gamma\left(1+\frac{d_n}{p}\right)}{\Gamma\left(1+\frac{d_n+2}{p}\right)}\right)^2
\ls \left(\frac{\Gamma\left(1+\frac{d_n}{p}\right)}{\Gamma\left(1+\frac{d_n+2}{p}\right)}\right)^2 -\, 
\frac{\Gamma\left(1+\frac{d_n}{p}\right)}{\Gamma\left(1+\frac{d_n+4}{p}\right)} 
\ls \frac{C_3}{p\,d_n} 
\,\left(\frac{\Gamma\left(1+\frac{d_n}{p}\right)}{\Gamma\left(1+\frac{d_n+2}{p}\right)}\right)^2$,
\end{enumerate}
where the $O$-notation in {\rm (a)} implies constants which may depend on $p$, $d_n$ and $q$, but which are in absolute value less than some absolute constant $c_1>0$,
and where $c_0, C_2, C_3$ are positive absolute constants. 
\end{lemma}
\begin{proof}
For part (a) we will use one of Binet's formulas for $\log \Gamma(x)$:
\[\log\Gamma(x) = \left(x-\frac{1}{2}\right)\log x - x + \frac{1}{2}\log(2\pi) + 2\int_0^\infty \frac{\arctan(t/x)}{e^{2\pi t} - 1}\,dt\]
for every positive $x$. Hence
\begin{align*}
&\log\frac{\Gamma\left(1+\frac{d_n}{p}\right)}{\Gamma\left(1+\frac{d_n+q}{p}\right)}
\\
=\, &\left(\frac{p+d_n}{p}-\frac{1}{2}\right)\log\left(\frac{p+d_n}{p}\right) - \frac{p+d_n}{p} + 2\int_0^\infty\frac{\arctan\bigl(pt/(p+d_n)\bigr)}{e^{2\pi t} - 1}\,dt
\\
&\quad - \left(\frac{p+d_n+q}{p}-\frac{1}{2}\right)\log\left(\frac{p+d_n+q}{p}\right) + \frac{p+d_n+q}{p} - 2\int_0^\infty\frac{\arctan\bigl(pt/(p+d_n+q)\bigr)}{e^{2\pi t} - 1}\,dt
\\
=\, & \frac{1}{2}\log\left(\frac{p+d_n+q}{p+d_n}\right) + \left[\frac{q}{p} -\frac{p+d_n}{p}\log\left(\frac{p+d_n+q}{p+d_n}\right)\right] - \frac{q}{p}\log\left(\frac{p+d_n+q}{p}\right)
\\
&\quad + 2\int_0^\infty\frac{\arctan\bigl(pt/(p+d_n)\bigr) - \arctan\bigl(pt/(p+d_n+q)\bigr)}{e^{2\pi t} - 1}\,dt.
\end{align*}
By a second-order Taylor approximation of the logarithmic funtion, we obtain
\begin{equation*} \log\left(\frac{p+d_n+q}{p+d_n}\right) = \log\left(1 + \frac{q}{p+d_n}\right) = \frac{q}{p+d_n} + O\left(\left(\frac{q}{p+d_n}\right)^2\right),\end{equation*}
and so
\[\frac{1}{2}\log\left(\frac{p+d_n+q}{p+d_n}\right) = O\left(\frac{q}{p+d_n}\right) = O\left(\frac{q}{n^2}\right),\]
and 
\[\frac{q}{p} -\frac{p+d_n}{p}\log\left(\frac{p+d_n+q}{p+d_n}\right) = O\left(\frac{q^2}{p(p+d_n)}\right).\]
On the other hand, by the mean value theorem we get, for every $t>0$,
\begin{align*}
\arctan\bigl(pt/(p+d_n)\bigr) &- \arctan\bigl(pt/(p+d_n+q)\bigr) = \frac{1}{1+(t/x_t)^2} \Big|\frac{pt}{p+d_n} - \frac{pt}{p+d_n+q}\Big|
\\
\intertext{for some $x_t\in \bigl[(p+d_n)/p, (p+d_n+q)/p\bigr]$, which makes the difference above}
& \ls \frac{(p+d_n+q)^2}{(p+d_n+q)^2 + (pt)^2} \frac{pqt}{(p+d_n)(p+d_n+q)}.
\end{align*}
Hence
\begin{align*} 
&\int_0^\infty\frac{\arctan\bigl(pt/(p+d_n)\bigr) - \arctan\bigl(pt/(p+d_n+q)\bigr)}{e^{2\pi t} - 1}\,dt
\\
\ls\ & \int_0^\infty \frac{\frac{pqt}{(p+d_n)(p+d_n+q)}}{2\pi t[1 + (2\pi t)/2! + (2\pi t)^2/3! + \cdots]} \,\frac{(p+d_n+q)^2}{(p+d_n+q)^2 + (pt)^2} 
\\
\ls\ & \int_0^\infty \frac{q}{2\pi(p+d_n)}\,\frac{1}{1+t^2}\,dt = O\left(\frac{q}{p+d_n}\right)
\end{align*}
Part (a) follows (given that $q/n^2 = O(1)$ for every $q\in [2,d_n]$, so $\exp\bigl(O(q/n^2)\bigr) = 1 + O(q/n^2))$.

\bigskip

For part (b) we will use Euler's infinite product representation for the Gamma function:
\begin{equation*} 
\Gamma(x) = \frac{1}{x} \prod_{l=1}^\infty\frac{\left(1+ \frac{1}{l}\right)^x}{1+\frac{x}{l}} 
\end{equation*}
valid for all $x>0$.
By this we obtain
\begin{align*} 
\frac{\Gamma\left(1+\frac{d_n}{p}\right)}{\Gamma\left(1+\frac{d_n+q}{p}\right)} &= 
\left(1 + \frac{q}{p+d_n}\right) \prod_{l=1}^\infty\left(1+\frac{q}{p(l+1) + d_n}\right)\, \left(1+ \frac{1}{l}\right)^{-q/p}
\\
& = \prod_{l=1}^\infty \left(1+\frac{q}{p l + d_n}\right) \, \left(1+ \frac{1}{l}\right)^{-q/p}
\end{align*}
for both $q=2$ and $q=4$, therefore
\begin{align*}
& \left(\frac{\Gamma\left(1+\frac{d_n}{p}\right)}{\Gamma\left(1+\frac{d_n+2}{p}\right)}\right)^2 - 
\frac{\Gamma\left(1+\frac{d_n}{p}\right)}{\Gamma\left(1+\frac{d_n+4}{p}\right)}  
\\
= \,& \prod_{l=1}^\infty \left(1+\frac{2}{p l + d_n}\right)^2 \, \left(1+ \frac{1}{l}\right)^{-4/p} - \prod_{l=1}^\infty \left(1+\frac{4}{p l + d_n}\right) \, \left(1+ \frac{1}{l}\right)^{-4/p}
\\
=\, & \prod_{l=1}^\infty \left(1+\frac{4}{p l + d_n} + \frac{4}{(pl+d_n)^2}\right) \, \left(1+ \frac{1}{l}\right)^{-4/p} - \prod_{l=1}^\infty \left(1+\frac{4}{p l + d_n}\right) \, \left(1+ \frac{1}{l}\right)^{-4/p}
\\
\ls \, & \sum_{l=1}^\infty \frac{4}{(pl+d_n)^2}\,\cdot  \prod_{l=1}^\infty \left(1+\frac{4}{p l + d_n} + \frac{4}{(pl+d_n)^2}\right) \, \left(1+ \frac{1}{l}\right)^{-4/p}.
\\
\ls\, & \int_0^\infty \frac{4}{(pt + d_n)^2}\,dt \,\cdot \left(\frac{\Gamma\left(1+\frac{d_n}{p}\right)}{\Gamma\left(1+\frac{d_n+2}{p}\right)}\right)^2.
\end{align*}
Similarly,
\begin{align*}
 \left(\frac{\Gamma\left(1+\frac{d_n}{p}\right)}{\Gamma\left(1+\frac{d_n+2}{p}\right)}\right)^2 - 
\frac{\Gamma\left(1+\frac{d_n}{p}\right)}{\Gamma\left(1+\frac{d_n+4}{p}\right)}  
&\gs \sum_{l=1}^\infty \frac{\frac{4}{(pl+d_n)^2}}{1+\frac{4}{pl+d_n}}\,\cdot  \prod_{l=1}^\infty \left(1+\frac{4}{p l + d_n} \right) \, \left(1+ \frac{1}{l}\right)^{-4/p}
\\
&\gtrsim \int_1^\infty \frac{4}{(pt + d_n)^2}\,dt\,\cdot \frac{\Gamma\left(1+\frac{d_n}{p}\right)}{\Gamma\left(1+\frac{d_n+4}{p}\right)}.
\end{align*}
Using also the conclusion of part (a), we deduce part (b).
\end{proof}

We are now ready to prove the following proposition (which is just a generalised version of Proposition \ref{prop:Var Kp and Mp for MnF}).

\begin{proposition} \label{prop:Var Kp and Mp} 
Let ${\rm Var}_{M_{a,b,c,p}}\bigl(\|x\|_2^2\bigr)$ $($or more briefly ${\rm Var}_{M_p}\bigl(\|x\|_2^2\bigr))$ denote the quantity
\begin{equation*} \frac{M_{a,b,c,p}\bigl(\|x\|_2^4\bigr)}{M_{a,b,c,p}(1)} - \left(\frac{M_{a,b,c,p}\bigl(\|x\|_2^2\bigr)}{M_{a,b,c,p}(1)}\right)^2,\end{equation*}
that is, the variance of the Euclidean norm with respect to the density $\exp\bigl(-\|x\|_p^p\bigr) f_{a,b,c}(x)$, where 
\begin{equation*} f_{a,b,c}(x) = \prod_{1\ls i<j\ls n}\big|\,x_i^a - x_j^a\big|^b\cdot \prod_{1\ls i\ls n}|x_i|^c.\end{equation*}
Let also $K_{p,E}$ denote the unit ball of the $p$-Schatten class in the subspace $E$ corresponding to the density $f_{a,b,c}$, 
and let
\begin{equation*} 
\sigma_{K_{p,E}}^2 = d_n\, 
\frac{{\rm Var}_{\overline{K}_{p,E}}\bigl(\|s(T)\|_2^2\bigr)}{\left[{\mathbb E}_{\overline{K}_{p,E}}\bigl(\|s(T)\|_2^2\bigr)\right]^2}.
\end{equation*}
The following relation is true:
\begin{equation} \label{eq:prop:Var Kp and Mp}
{\rm Var}_{M_{a,b,c,p}}\bigl(\|x\|_2^2\bigr) \simeq \max\Bigl\{\sigma_{K_{p,E}}^2, \frac{1}{p}\Bigr\}\cdot n^{4/p},\end{equation}
where the implied constants depend only on the integers $a,b,c$.
\begin{remark} \label{rem:Var Kp and Mp}
The proposition shows that $K_{p,E}$ satisfies the thin-shell conjecture, in other words, 
$\sigma_{K_{p,E}}^2 \ls C$ for some absolute constant $C$, 
if and only if ${\rm Var}_{M_{a,b,c,p}}\bigl(\|x\|_2^2\bigr) \ls C^\prime n^{4/p}$ (for some other absolute constant depending linearly on $C$).

In fact, using the best bounds for $\sigma_{K_{p,E}}$ that we currently have, 
which are due to Barthe and Cordero-Erausquin \cite{Barthe-Cordero-2013}, 
we can infer the following estimate right now:
for all $p$ and for all subspaces $E$ for which we know that $K_{p,E}$ is in isotropic position 
(these include ${\cal M}_n({\mathbb R})$, ${\cal M}_n({\mathbb C})$ and ${\cal M}_n({\mathbb H})$),
 \begin{equation} {\rm Var}_{M_p}\bigl(\|x\|_2^2\bigr) \lesssim n\cdot n^{4/p} = o(n^2)\cdot n^{4/p}.\end{equation}
 We are going to take advantage of this estimate in the sequel.
\end{remark}
\begin{remark}
Note that in \eqref{eq:prop:Var Kp and Mp} the estimate
\begin{equation}\label{eq:rem:prop:Var Kp and Mp} {\mathbb E}_{\overline{K}_{p,E}}\bigl(\|s(T)\|_2^2\bigr) \simeq d_n,\end{equation}
which follows from the arguments of \cite{Konig-Meyer-Pajor-1998} and \cite{Guedon-Paouris-2007},
and is valid for all of the classical subspaces $E$ we have mentioned, is already incorporated. If we 
do not use it yet, then, as will be clear from the ensuing proof, we will get
\begin{align*} {\rm Var}_{M_p}\bigl(\|x\|_2^2\bigr) 
&\simeq \max\Bigl\{\sigma_{K_{p,E}}^2, \frac{1}{p}\Bigr\}\cdot d_n^{4/p-1}\, |K_{p,E}|^{4/d_n}\, \left[{\mathbb E}_{\overline{K}_{p,E}}\bigl(\|s(T)\|_2^2\bigr)\right]^2
\\
&\simeq \max\Bigl\{\sigma_{K_{p,E}}^2, \frac{1}{p}\Bigr\}\cdot \frac{1}{d_n^2}\,\left[{\mathbb E}_{\overline{K}_{p,E}}\bigl(\|s(T)\|_2^2\bigr)\right]^2\cdot n^{4/p},
\end{align*}
where the last equivalence follows from the volume estimates for 
the unit balls of the Schatten classes provided by \cite[Proposition 3]{Guedon-Paouris-2007} 
(see end of Section \ref{sec:Notation}).
\end{remark}
\begin{proof}
Note that by \eqref{eq:rem:prop:Var Kp and Mp} we have
\begin{equation*} {\rm Var}_{\overline{K}_{p,E}}\bigl(\|s(T)\|_2^2\bigr) 
= \frac{1}{d_n} \left[{\mathbb E}_{\overline{K}_{p,E}}\bigl(\|s(T)\|_2^2\bigr)\right]^2 \cdot \sigma_{K_{p,E}}^2 \simeq \sigma_{K_{p,E}}^2\cdot d_n.
\end{equation*}
But
\begin{align*}
{\rm Var}_{\overline{K}_{p,E}}\bigl(\|s(T)\|_2^2\bigr) 
&= \int_{\overline{K}_{p,E}} \|s(T)\|_2^4\,dT - 
\left(\int_{\overline{K}_{p,E}} \|s(T)\|_2^2\,dT\right)^2
\\
& = \frac{1}{|K_{p,E}|^{1+\frac{4}{d_n}}} \int_{K_{p,E}} \|s(T)\|_2^4\,dT 
 -  \left(\frac{1}{|K_{p,E}|^{1+\frac{2}{d_n}}}\int_{K_{p,E}} \|s(T)\|_2^2\,dT\right)^2
 \\
 & = \frac{1}{|K_{p,E}|^{4/d_n}}\,\left[\frac{1}{|K_{p,E}|} \int_{K_{p,E}} \|s(T)\|_2^4\,dT - \left(\frac{1}{|K_{p,E}|}\int_{K_{p,E}} \|s(T)\|_2^2\,dT\right)^2\right].
\end{align*}
Given that $|K_{p,E}|^{4/d_n} \simeq d_n^{-1-2/p}$, we therefore obtain that
\begin{align}
\nonumber
\sigma_{K_{p,E}}^2\cdot d_n^{-2/p} \ &\simeq\  |K_{p,E}|^{4/d_n}\cdot {\rm Var}_{\overline{K}_{p,E}}\bigl(\|s(T)\|_2^2\bigr) 
\\ \nonumber
& =\ \frac{1}{|K_{p,E}|} \int_{K_{p,E}} \|s(T)\|_2^4\,dT - \left(\frac{1}{|K_{p,E}|}\int_{K_{p,E}} \|s(T)\|_2^2\,dT\right)^2,
\\
\intertext{which, by use of Lemmas \ref{lem:reduction to n-integrals1} and its counterparts for the other subspaces,
becomes}
\nonumber
&=\ \frac{\Gamma\left(1+\frac{d_n}{p}\right)}{\Gamma\left(1+\frac{d_n+4}{p}\right)}\,\frac{M_p\bigl(\|x\|_2^4\bigr)}{M_p(1)} - 
\left(\frac{\Gamma\left(1+\frac{d_n}{p}\right)}{\Gamma\left(1+\frac{d_n+2}{p}\right)}\right)^2\,\left(\frac{M_p\bigl(\|x\|_2^2\bigr)}{M_p(1)}\right)^2
\\ \label{eqp1:prop:Var Kp and Mp} 
 &=\ \frac{\Gamma\left(1+\frac{d_n}{p}\right)}{\Gamma\left(1+\frac{d_n+4}{p}\right)}\,\left[\frac{M_p\bigl(\|x\|_2^4\bigr)}{M_p(1)} - \left(\frac{M_p\bigl(\|x\|_2^2\bigr)}{M_p(1)}\right)^2\right] 
 \\ \label{eqp2:prop:Var Kp and Mp} 
 &\qquad\qquad -  \left[\left(\frac{\Gamma\left(1+\frac{d_n}{p}\right)}{\Gamma\left(1+\frac{d_n+2}{p}\right)}\right)^2- 
 \frac{\Gamma\left(1+\frac{d_n}{p}\right)}{\Gamma\left(1+\frac{d_n+4}{p}\right)}\right]\,\left(\frac{M_p\bigl(\|x\|_2^2\bigr)}{M_p(1)}\right)^2.
\end{align}
Now we use the estimates of Lemma \ref{lem:Gamma estimates} (without needing yet the more accurate form in which we stated them): 
the term in \eqref{eqp1:prop:Var Kp and Mp} can be rewritten as
\begin{equation*}
\frac{\Gamma\left(1+\frac{d_n}{p}\right)}{\Gamma\left(1+\frac{d_n+4}{p}\right)}\ {\rm Var}_{M_p}\bigl(\|x\|_2^2\bigr) \simeq d_n^{-4/p} \cdot {\rm Var}_{M_p}\bigl(\|x\|_2^2\bigr),
\end{equation*}
and since the term in \eqref{eqp2:prop:Var Kp and Mp} is negative, we get
\begin{equation} \label{eqp3:prop:Var Kp and Mp}
{\rm Var}_{M_p}\bigl(\|x\|_2^2\bigr) \gtrsim \sigma_{K_{p,E}}^2\cdot d_n^{2/p}\simeq \sigma_{K_{p,E}}^2\cdot n^{4/p}.\end{equation}
In addition, since the sum of the terms in \eqref{eqp1:prop:Var Kp and Mp} and \eqref{eqp2:prop:Var Kp and Mp} is equal to a positive quantity, we obtain
\begin{align*}
d_n^{-4/p} \cdot {\rm Var}_{M_p}\bigl(\|x\|_2^2\bigr)
&\gtrsim \left[\left(\frac{\Gamma\left(1+\frac{d_n}{p}\right)}{\Gamma\left(1+\frac{d_n+2}{p}\right)}\right)^2- 
 \frac{\Gamma\left(1+\frac{d_n}{p}\right)}{\Gamma\left(1+\frac{d_n+4}{p}\right)}\right]\,\left(\frac{M_p\bigl(\|x\|_2^2\bigr)}{M_p(1)}\right)^2
 \\
 &\gs \frac{C_2}{p(p+d_n)}\, \left(\frac{\Gamma\left(1+\frac{d_n}{p}\right)}{\Gamma\left(1+\frac{d_n+2}{p}\right)}\right)^2 \,\left(\frac{M_p\bigl(\|x\|_2^2\bigr)}{M_p(1)}\right)^2
 \\
 & = \frac{C_2}{p(p+d_n)}\ |K_{p,E}|^{4/d_n}\,\left[{\mathbb E}_{\overline{K}_{p,E}}\bigl(\|s(T)\|_2^2\bigr)\right]^2
 \\
 &\simeq \frac{C_2}{p(p+d_n)}\,d_n^{1-2/p}.
\end{align*}
This shows that ${\rm Var}_{M_p}\bigl(\|x\|_2^2\bigr)\gtrsim \frac{1}{p}\cdot d_n^{2/p} \simeq \frac{1}{p}\cdot n^{4/p}$ if $1\ls p\lesssim d_n$, whereas if $p\gtrsim d_n$
then $\max\{\sigma_{K_{p,E}}^2, 1/p\} = \sigma_{K_{p,E}}^2$ given that for every 
centred convex body in a $d_n$-dimensional space we have
\begin{equation*} \sigma_{K_{p,E}}^2 \gs \sigma_{B_2^{d_n}}^2 = \frac{4}{d_n+4}\end{equation*}
(see \cite[Theorem 2]{Bobkov-Koldobsky-2003}). Combining with \eqref{eqp3:prop:Var Kp and Mp}, we conclude that
\begin{equation}\label{eqp4:prop:Var Kp and Mp} {\rm Var}_{M_p}\bigl(\|x\|_2^2\bigr)\gtrsim \max\Bigl\{\sigma_{K_{p,E}}^2,\frac{1}{p}\Bigr\}\cdot n^{4/p}.\end{equation}

\smallskip

In the opposite direction, we have
\begin{align*}
d_n^{-4/p} \cdot {\rm Var}_{M_p}\bigl(\|x\|_2^2\bigr) &\lesssim \sigma_{K_{p,E}}^2\cdot d_n^{-2/p} +
\left[\left(\frac{\Gamma\left(1+\frac{d_n}{p}\right)}{\Gamma\left(1+\frac{d_n+2}{p}\right)}\right)^2- 
 \frac{\Gamma\left(1+\frac{d_n}{p}\right)}{\Gamma\left(1+\frac{d_n+4}{p}\right)}\right]\,\left(\frac{M_p\bigl(\|x\|_2^2\bigr)}{M_p(1)}\right)^2
 \\
 &\ls \sigma_{K_{p,E}}^2\cdot d_n^{-2/p} + \frac{C_3}{p\,d_n}\  |K_{p,E}|^{4/d_n}\,\left[{\mathbb E}_{\overline{K}_{p,E}}\bigl(\|s(T)\|_2^2\bigr)\right]^2
 \\
 &\simeq \sigma_{K_{p,E}}^2\cdot d_n^{-2/p} + \frac{C_3}{p\,d_n}\,d_n^{1-2/p},
\end{align*}
whence we obtain the reverse inequality to \eqref{eqp4:prop:Var Kp and Mp}.
This completes the proof of the proposition.
\end{proof}
\end{proposition}

\bigskip

As mentioned in the Introduction, our task now becomes to find good estimates for the quantity ${\rm Var}_{M_p}\bigl(\|x\|_2^2\bigr)$,
and ideally to show that it is $O(n^{4/p})$. One of our tools towards this goal is Lemma \ref{lem:int-by-parts} that was stated in the Introduction;
it will become apparent that one other thing we need is to be able to relate integrals of the form $M_p\bigl(\|x\|_p^l f(x)\bigr)$, where $l$ is some
real number, to each other.

\begin{lemma}\label{lem:multiples of p-norm}
Let $l,s\in {\mathbb R}$ be such that $s > -d_n$ and $l+s > -d_n$. Suppose also that $f: {\mathbb R}^n\setminus\{0\} \to {\mathbb R}$ is 
a continuous, $s$-homogeneous function. Then
\begin{equation*} M_p\bigl(\|x\|_p^l f(x)\bigr) = \frac{\Gamma\left(\frac{d_n+l+s}{p}\right)}{\Gamma\left(\frac{d_n+s}{p}\right)} M_p(f).\end{equation*}
\end{lemma}
\begin{proof}
We use a polar integration formula. Since both $f$ and $\|x\|_p^l\, f(x)$ are positively homogeneous functions of order $s$ and $l+s$ respectively,
we have
\begin{align*} 
M_p(f) &= \int_{{\mathbb R}^n} f(x) \cdot f_{a,b,c}(x)\exp\bigl(-\|x\|_p^p\bigr)\,dx
\\
&= n\,{\rm vol}(B_p^n) \int_0^\infty r^{d_n +s -1} e^{-r^p} \int_{\partial B_p^n} f(u)f_{a,b,c}(u)\,d\mu_{B_p^n}(u)\,dr
\\
& = n\, {\rm vol}(B_p^n)\,\frac{\Gamma\left(\frac{d_n+s}{p}\right)}{p}\,\int_{\partial B_p^n} f(u)f_{a,b,c}(u)\,d\mu_{B_p^n}(u),
\end{align*}
and similarly
\begin{align*} 
M_p\bigl(\|x\|_p^l f(x)\bigr) &= \int_{{\mathbb R}^n} \|x\|_p^l\, f(x) \cdot f_{a,b,c}(x)\exp\bigl(-\|x\|_p^p\bigr)\,dx
\\
&= n\, {\rm vol}(B_p^n) \int_0^\infty r^{d_n +l + s -1} e^{-r^p} \int_{\partial B_p^n} \|u\|_p^l\,f(u)f_{a,b,c}(u)\,d\mu_{B_p^n}(u)\,dr
\\
& = n\, {\rm vol}(B_p^n)\,\frac{\Gamma\left(\frac{d_n+l+s}{p}\right)}{p}\,\int_{\partial B_p^n} f(u)f_{a,b,c}(u)\,d\mu_{B_p^n}(u),
\end{align*}
where $\mu_{B_p^n}$ is a type of cone-volume measure (see e.g. \cite{Naor-Romik-2003})
on the boundary $\partial B_p^n$ of $B_p^n$, which is defined by
\begin{equation*} \mu_{B_p^n}(A) := \frac{|\{tu: u\in A, 0\ls t\ls 1\}|}{|B_p^n|}.\end{equation*}
This completes the proof.
\end{proof}

Note that the case $l=p$ also follows from Lemma \ref{lem:int-by-parts} applied with $\xi = 0$ 
(and at first with functions $f$ that satisfy the hypotheses 
of the lemma; see also \cite[Corollary 7(a)]{Guedon-Paouris-2007} 
for a different proof of the case $l=p$ that works directly for arbitrary continuous functions). 
However, in what follows, we will need to make use of other cases of Lemma \ref{lem:multiples of p-norm} too.

\subsection{Orders of magnitude of relevant quantities}

Given that, by symmetry,
\begin{align} 
\nonumber
{\rm Var}_{M_p}\bigl(\|x\|_2^2\bigr) &= \frac{M_p\bigl(\|x\|_2^4\bigr)}{M_p(1)} - \left(\frac{M_p\bigl(\|x\|_2^2\bigr)}{M_p(1)}\right)^2 
\\ \label{eq:relevant-quantities}
& = n\cdot  \frac{M_p\bigl(x_1^4\bigr)}{M_p(1)} + n(n-1) \cdot \frac{M_p\bigl(x_1^2x_2^2\bigr)}{M_p(1)} - n^2\cdot \left(\frac{M_p\bigl(x_1^2\bigr)}{M_p(1)}\right)^2 
\\ \nonumber
& = n\cdot {\rm Var}_{M_p}\bigl(x_1^2\bigr) + n(n-1)\cdot \left[\frac{M_p\bigl(x_1^2x_2^2\bigr)}{M_p(1)} - \left(\frac{M_p\bigl(x_1^2\bigr)}{M_p(1)}\right)^2\right],
\end{align}
our first objective thus becomes to study the order of magnitude of the quantities $M_p(x_1^2)/M_p(1)$, $M_p(x_1^4)/M_p(1)$ and $M_p(x_1^2x_2^2)/M_p(1)$.

To this end we will use Lemma \ref{lem:int-by-parts} that was stated in the Introduction, 
and has been used for the exact same purpose both in \cite{Konig-Meyer-Pajor-1998} and in \cite{Guedon-Paouris-2007}.

\begin{remark}
In \cite{Konig-Meyer-Pajor-1998} the authors apply Lemma \ref{lem:int-by-parts} with $\xi = 0$ or $\xi = p$ and with $f(x) =1$.
On the other hand, the authors of \cite{Guedon-Paouris-2007} have to apply the lemma in more general cases as well:
they obtain recursive equivalences using the lemma with $\xi = 2$ or $\xi = p$ and 
with $f$ being different powers of the Euclidean norm.

In both situations, it turns out that the most 
bothersome to deal with term in \eqref{eq1:lem:int-by-parts} 
is the last one:
the way they estimate it in both of the abovementioned papers is by observing that
\begin{equation}\label{eq:equiv-for-last-term} 
\zeta_1(a,\xi)\cdot \bigl(|x_i|^\xi + |x_j|^\xi\bigr) \ls \frac{|x_i|^\xi\,x_i^a - |x_j|^\xi\,x_j^a}{x_i^a - x_j^a} 
\ls \zeta_2(a,\xi)\cdot \bigl(|x_i|^\xi + |x_j|^\xi\bigr)
\end{equation}
for all $x_i\neq x_j$, where $\zeta_1(a,\xi) = \min\bigl\{\frac{1}{2},\frac{a+\xi}{2a}\bigr\}$ (or, if $a$ is even, 
$\zeta_1(a,\xi) = \min\bigl\{1,\frac{a+\xi}{2a}\bigr\}$), and $\zeta_2(a,\xi) = \max\bigl\{1,\frac{a+\xi}{2a}\bigr\}$, 
and then by writing 
\begin{align*} M_p\left(f(x)\sum_{i=1}^n\sum_{j\neq i}\frac{|x_i|^\xi\,x_i^a}{x_i^a - x_j^a}\right)
&= M_p\left(f(x)\sum_{i<j}\frac{|x_i|^\xi\,x_i^a - |x_j|^\xi\,x_j^a}{x_i^a - x_j^a}\right) 
\\
&\simeq \zeta_i(a,\xi)\cdot M_p\left(f(x)\sum_{i<j}\Bigl(|x_i|^\xi + |x_j|^\xi\Bigr)\right) 
\\
&= \zeta_i(a,\xi)\cdot (n-1)\,M_p\left(f(x)\|x\|_\xi^\xi\right).
\end{align*}
In this way Lemma \ref{lem:int-by-parts} leads to
\begin{align} 
\label{eq:conclusion of lem:int-by-parts}
\left(\zeta_1(a,\xi)\, \frac{d_n}{n} + \xi\right) M_p\left(f(x)\|x\|_\xi^\xi\right) 
&\lesssim pM_p\left(\|x\|_{\xi + p}^{\xi + p}f(x)\right) - M_p\left(\sum_{i=1}^n|x_i|^\xi x_i\frac{\partial f}{\partial x_i}(x)\right) 
\\ \nonumber
&\lesssim \zeta_2(a,\xi)\, \frac{d_n}{n} M_p\left(f(x)\|x\|_\xi^\xi\right)
\end{align}
for any positive function $f$ satisfying the assumptions of the lemma. 
\end{remark}

\begin{proposition}\label{prop:order-of-magnitudes}
Let $M_p$ denote integration over ${\mathbb R}^n$ with respect to one of the densities $f_{a,b,c,p}$ of the form
\begin{equation*} 
f_{a,b,c,p}(x) = \exp\bigl(-\|x\|_p^p\bigr)\cdot\prod_{1\ls i<j\ls n}\big|\,x_i^a - x_j^a\big|^b\cdot \prod_{1\ls i\ls n}|x_i|^c,
\end{equation*}
where $p\in [1,+\infty)$ and $a,b$ are positive integers, $c$ is a non-negative integer. We have
\begin{equation}\label{eq:prop:order-of-magnitudes}  \frac{M_p\bigl(x_1^2\bigr)}{M_p(1)}\simeq n^{2/p} \quad\quad\hbox{and}\quad\quad
\frac{M_p\bigl(x_1^4\bigr)}{M_p(1)} \simeq \frac{M_p\bigl(x_1^2x_2^2\bigr)}{M_p(1)} \simeq n^{4/p}.
\end{equation}
\end{proposition}
\begin{proof}
The first part of \eqref{eq:prop:order-of-magnitudes} is essentially the core result of \cite{Konig-Meyer-Pajor-1998}.
For the reader's convenience, let us recall how one can easily deduce it from \eqref{eq:conclusion of lem:int-by-parts} with the help
of Hölder's inequality and Lemma \ref{lem:multiples of p-norm}: we first apply \eqref{eq:conclusion of lem:int-by-parts} 
with $\xi=2$ and $f(x) ={\bf 1}$
to obtain that
\begin{align*}
\frac{d_n}{n}\,M_p\bigl(\|x\|_2^2\bigr) \gtrsim p\, M_p\bigl(\|x\|_{p+2}^{p+2}\bigr) \gs \frac{p}{n^{2/p}}\, M_p\bigl(\|x\|_p^{p+2}\bigr) 
\simeq  \frac{d_n^{2/p}}{n^{2/p}}\, d_n\, M_p(1),
\end{align*}
or in other words that $M_p(x_1^2)/M_p(1) \gtrsim n^{2/p}$.
To also get the reverse inequality, we apply \eqref{eq:conclusion of lem:int-by-parts} with $\xi = p$ and $f(x) ={\bf 1}$: this gives 
\begin{equation*} 
\frac{a + p}{2a}\,\frac{d_n}{n}\, M_p\bigl(\|x\|_p^p\bigr) \gtrsim p\,M_p\bigl(\|x\|_{2p}^{2p}\bigr),
\end{equation*}
and then, by a simple application of Hölder's inequality, we can conclude that
\begin{equation*} 
\frac{M_p\bigl(x_1^2\bigr)}{M_p(1)} \ls \left(\frac{M_p\bigl(|x_1|^{2p}\bigr)}{M_p(1)}\right)^{1/p}
\lesssim \left(n\frac{M_p\bigl(|x_1|^p\bigr)}{M_p(1)}\right)^{1/p} = \biggl(\frac{d_n}{p}\biggr)^{1/p}.
\end{equation*}

The second part of \eqref{eq:prop:order-of-magnitudes} can follow by very similar reasoning: 
in this case we have to apply \eqref{eq:conclusion of lem:int-by-parts}
with $\xi = 2p$ or $\xi = 3p$ as well, to be able to compare $M_p(x_1^4)/M_p(1)$ to $M_p(|x_1|^{4p})/M_p(1)$.

Finally note that
\begin{equation*} \frac{M_p\bigl(x_1^2x_2^2\bigr)}{M_p(1)} \ls 
\sqrt{\frac{M_p\bigl(x_1^4\bigr)}{M_p(1)}\,\frac{M_p\bigl(x_2^4\bigr)}{M_p(1)}},\end{equation*}
while
\begin{multline*} 
n(n-1) \frac{M_p\bigl(x_1^2x_2^2\bigr)}{M_p(1)}  = \frac{M_p\bigl(\|x\|_2^4\bigr)}{M_p(1)} - n\frac{M_p\bigl(x_1^4\bigr)}{M_p(1)}
 \gs \left(\frac{M_p\bigl(\|x\|_2^2\bigr)}{M_p(1)} \right)^2 - n\frac{M_p\bigl(x_1^4\bigr)}{M_p(1)}
 \\
 \gs c_1n^2\cdot n^{4/p} - c_2n\cdot n^{4/p} \simeq n^2\cdot n^{4/p}.
\end{multline*}
This completes the proof.
\end{proof}

\section{Proof of the main results}\label{sec:main-proofs}

It is not difficult to convince ourselves that estimating the variance of the Euclidean norm with respect to the densities $f_{a,b,c,p}$ is
a more delicate problem than merely finding the orders of magnitude of the terms $\bigl(M_p(x_1^2)/M_p(1)\bigr)^2$, $M_p(x_1^4)/M_p(1)$ and $M_p(x_1^2x_2^2)/M_p(1)$
appearing when we write out the variance: we have just seen that they are all $\simeq n^{4/p}$, however it is obvious that we cannot extract any non-trivial information
about the variance from these equivalences if we do not also find a way to estimate the constants appearing in them (or in other words, the 
coefficient of $n^{4/p}$ in each case). Thus, we will now attempt to estimate the contribution of the last term in \eqref{eq1:lem:int-by-parts} 
in a more precise manner: one way this can be done is through the following proposition.

\begin{proposition}\label{prop:int-by-parts-&-sym}
Let $\xi\gs 0$ and $s> -d_n-\xi$, and let $f:{\mathbb R}^n\setminus\{0\} \to {\mathbb R}^+$ be an $s$-homogeneous function
that satisfies the hypotheses of Lemma $\ref{lem:int-by-parts}$. Suppose moreover that $f$ is a symmetric function. Then we have
\begin{multline} \label{eq1:prop:int-by-parts-&-sym}
\left(\frac{2d_n + (\xi-c-1)n}{n}\right) M_p\bigl(\|x\|_\xi^\xi f(x)\bigr) =
\\p M_p\bigl( \|x\|_{\xi+p}^{\xi+p}\, f(x)\bigr) - M_p\biggl(\sum_{i=1}^n |x_i|^\xi x_i \frac{\partial f}{\partial x_i}(x)\biggr)
+ ab n(n-1) M_p \biggl(\frac{|x_2|^\xi x_1^a}{x_1^a - x_2^a} f(x)\biggr),
\end{multline}
where by symmetry we can also write 
\begin{equation*}
M_p \biggl(\frac{|x_2|^\xi x_1^a}{x_1^a - x_2^a} f(x)\biggr) = \dfrac{1}{2}M_p\biggl(\frac{x_1^a|x_2|^\xi - x_2^a|x_1|^\xi}{x_1^a - x_2^a} f(x)\biggr).
\end{equation*}
If in addition $a$ is an even integer, then
\begin{equation} \label{eq2:prop:int-by-parts-&-sym}
M_p \biggl(\frac{|x_2|^\xi x_1^a}{x_1^a - x_2^a} f(x)\biggr) 
= 
\left\{\begin{array}{cl} \dfrac{1}{2}M_p\biggl(\dfrac{|x_1|^{a-\xi} - |x_2|^{a - \xi}}{|x_1|^a - |x_2|^a}\,|x_1x_2|^\xi f(x)\biggr) > 0 & \hbox{if}\ \xi < a 
\\ \\
0 &  \hbox{if}\ \xi = a 
\\ \\
\dfrac{1}{2}M_p\biggl(\dfrac{|x_2|^{\xi-a} - |x_1|^{\xi-a}}{|x_1|^a - |x_2|^a}\,|x_1x_2|^a f(x)\biggr) < 0 & \hbox{if}\ \xi > a.
\end{array}\right.. 
\end{equation}
\end{proposition}
\begin{proof}
We first prove \eqref{eq1:prop:int-by-parts-&-sym}. By Lemma \ref{lem:int-by-parts} we can write
\begin{multline} \label{eqp1:prop:int-by-parts-&-sym}
(\xi + c +1)M_p\left(f(x)\sum_{i=1}^n|x_i|^\xi\right) = 
\\
pM_p\left(\|x\|_{\xi + p}^{\xi + p}f(x)\right) - M_p\left(\sum_{i=1}^n|x_i|^\xi x_i\frac{\partial f}{\partial x_i}(x)\right) - 
abM_p\left(f(x)\sum_{i=1}^n\sum_{j\neq i}\frac{|x_i|^\xi\,x_i^a}{x_i^a - x_j^a}\right).
\end{multline}
Note that by symmetry the last summand is equal to
\begin{equation*} abn(n-1)M_p \biggl(\frac{|x_1|^\xi\, x_1^a}{x_1^a - x_2^a} f(x)\biggr),\end{equation*}
which we can rewrite as
\begin{equation*} 
M_p \biggl(\frac{\bigl(|x_1|^\xi + |x_2|^\xi\bigr)x_1^a}{x_1^a - x_2^a} f(x)\biggr)-M_p \biggl(\frac{|x_2|^\xi\, x_1^a}{x_1^a - x_2^a} f(x)\biggr).\end{equation*}
Since the function $\bigl(|x_1|^\xi + |x_2|^\xi\bigr)f(x)$ is invariant under permuting the first two coordinates, it follows that
\begin{align*} 
M_p \biggl(\frac{x_1^a}{x_1^a - x_2^a} \bigl(|x_1|^\xi + |x_2|^\xi\bigr)f(x)\biggr) &= \frac{1}{2} M_p \biggl(\frac{x_1^a - x_2^a}{x_1^a - x_2^a} \bigl(|x_1|^\xi + |x_2|^\xi\bigr)f(x)\biggr)
\\
&= \frac{1}{2} M_p \bigl(\bigl(|x_1|^\xi + |x_2|^\xi\bigr)f(x)\bigr) = \frac{1}{n} M_p\bigl(\|x\|_\xi^\xi f(x)\bigr).
\end{align*}
We thus conclude that the last summand in \eqref{eqp1:prop:int-by-parts-&-sym} is equal to 
\begin{multline*} ab(n-1) M_p\bigl(\|x\|_\xi^\xi f(x)\bigr) - abn(n-1) M_p \biggl(\frac{|x_2|^\xi\, x_1^a}{x_1^a - x_2^a} f(x)\biggr) 
\\
= \left(\frac{2(d_n - (c+1)n)}{n}\right) M_p\bigl(\|x\|_\xi^\xi f(x)\bigr) - abn(n-1) M_p \biggl(\frac{|x_2|^\xi\, x_1^a}{x_1^a - x_2^a} f(x)\biggr).
\end{multline*}
This gives \eqref{eq1:prop:int-by-parts-&-sym}. The other two equations follow by symmetry and, in the case of \eqref{eq2:prop:int-by-parts-&-sym},
by the fact that $x_i^a = |x_i|^a$ when $a$ is an even integer. This completes the proof.
\end{proof}

\medskip

The following corollary summarises the three main identities that Proposition \ref{prop:int-by-parts-&-sym} 
gives us for densities of the form $f_{2,b,c,p}$ and which we will need to prove Theorems \ref{thm:thin-shell-conj} and \ref{thm:neg-cor-prop}.

\begin{corollary}
Let $M_p$ denote integration with respect to a density of the form 
$f_{2,b,c,p} = \prod_{i<j}\big|\,x_i^2 - x_j^2\big|^b\cdot \prod_i|x_i|^c\,e^{-\|x\|_p^p}$
$($namely, let $a=2)$.
By applying Proposition $\ref{prop:int-by-parts-&-sym}$ with $\xi = 2$ and $f(x) = {\bf 1}$ we obtain
\begin{equation}\label{eq:identity1} 
\left(\frac{2d_n + (1-c)n}{n}\right) \frac{M_p\bigl(\|x\|_2^2\bigr)}{M_p(1)} = p \frac{M_p\bigl(\|x\|_{p+2}^{p+2}\bigr)}{M_p(1)}. 
\end{equation}
By applying Proposition $\ref{prop:int-by-parts-&-sym}$ with $\xi = 2$ and $f(x) = \|x\|_2^2$ we obtain
\begin{equation}\label{eq:identity2}
\left(\frac{2d_n + (1-c)n}{n}\right) \frac{M_p\bigl(\|x\|_2^4\bigr)}{M_p(1)} = p \frac{M_p\bigl(\|x\|_2^2\cdot \|x\|_{p+2}^{p+2}\bigr)}{M_p(1)} -
2\frac{M_p\bigl(\|x\|_4^4\bigr)}{M_p(1)}. 
\end{equation}
Finally, by applying Proposition $\ref{prop:int-by-parts-&-sym}$ with $\xi = 4$ and $f(x) = {\bf 1}$ we obtain
\begin{equation}\label{eq:identity3}
\left(\frac{2d_n + (3-c)n}{n}\right) \frac{M_p\bigl(\|x\|_4^4\bigr)}{M_p(1)} = p \frac{M_p\bigl(\|x\|_{p+4}^{p+4}\bigr)}{M_p(1)} 
- \,\bigl(d_n\, -\,  (c+1)n\bigr)\frac{M_p\bigl(x_1^2x_2^2\bigr)}{M_p(1)}.
\end{equation} 
\end{corollary}

\medskip

\subsection{On the cases of Schatten classes corresponding to $a=2$ and to large $p$}

Here we prove Theorem \ref{thm:thin-shell-conj}. We will combine identities \eqref{eq:identity1} and \eqref{eq:identity2} 
with a simple application of Hölder's inequality, by which we have
\begin{equation}\label{eq:Holder} 
\|x\|_p^{p+2}\gs \|x\|_{p+2}^{p+2} \gs \|x\|_p^{p+2}\cdot n^{-2/p} = \left(1 - O\left(\frac{\log n}{p}\right)\right)\|x\|_p^{p+2}
\end{equation} 
for every $p>>\log n$. Indeed, by \eqref{eq:identity1}, \eqref{eq:Holder} and Lemma \ref{lem:multiples of p-norm} we obtain that
\begin{align} 
\label{eqp1:thm:thin-shell-conj}
\left(\frac{2d_n + (1-c)n}{n}\right) \frac{M_p\bigl(\|x\|_2^2\bigr)}{M_p(1)} &= p \frac{M_p\bigl(\|x\|_{p+2}^{p+2}\bigr)}{M_p(1)} 
\\ \nonumber
&\gs \left(1 - O\left(\frac{\log n}{p}\right)\right) p\frac{M_p\bigl(\|x\|_p^{p+2}\bigr)}{M_p(1)} 
\\ \nonumber
& = \left(1 - O\left(\frac{\log n}{p}\right)\right) d_n\frac{\Gamma\left(1+\frac{d_n+2}{p}\right)}{\Gamma\left(1+\frac{d_n}{p}\right)}.
\end{align}
We obviously also have
\begin{equation*} 
\frac{M_p\bigl(\|x\|_4^4\bigr)}{M_p(1)} = n \frac{M_p\bigl(x_1^4\bigr)}{M_p(1)} \gs n\left(\frac{M_p\bigl(x_1^2\bigr)}{M_p(1)}\right)^2 
= \frac{1}{n} \left(\frac{M_p\bigl(\|x\|_2^2\bigr)}{M_p(1)}\right)^2.
\end{equation*}
In view of the above estimates, as well as of part (a) of Lemma \ref{lem:Gamma estimates}, \eqref{eq:identity2} and Lemma \ref{lem:multiples of p-norm} now give
\begin{align*} 
\left(\frac{2d_n + (1-c)n}{n}\right) \frac{M_p\bigl(\|x\|_2^4\bigr)}{M_p(1)} & = p \frac{M_p\bigl(\|x\|_2^2\cdot \|x\|_{p+2}^{p+2}\bigr)}{M_p(1)} -
2\frac{M_p\bigl(\|x\|_4^4\bigr)}{M_p(1)} 
\\
& \ls p \frac{M_p\bigl(\|x\|_2^2\cdot \|x\|_p^{p+2}\bigr)}{M_p(1)} - \frac{2}{n}\left(\frac{M_p\bigl(\|x\|_2^2\bigr)}{M_p(1)}\right)^2
\\
&= (d_n+2)\frac{\Gamma\left(1+\frac{d_n+4}{p}\right)}{\Gamma\left(1+\frac{d_n+2}{p}\right)}\frac{M_p\bigl(\|x\|_2^2\bigr)}{M_p(1)} - \frac{2}{n}\left(\frac{M_p\bigl(\|x\|_2^2\bigr)}{M_p(1)}\right)^2
\\
& = \left(1+ O\left(\frac{\log n}{p}\right) + O\left(\frac{1}{n^2}\right)\right)\left(\frac{2d_n + (1-c)n}{n}\right) \left(\frac{M_p\bigl(\|x\|_2^2\bigr)}{M_p(1)} \right)^2.
\end{align*}

Thus, we conclude that
\begin{equation*} {\rm Var}_{M_p}\bigl(\|x\|_2^2\bigr) = \left(O\left(\frac{\log n}{p}\right) + O\left(\frac{1}{n^2}\right)\right)\left(\frac{M_p\bigl(\|x\|_2^2\bigr)}{M_p(1)} \right)^2
 \ls C\max\Bigl\{\frac{n^2\log n}{p}, 1\Bigr\}\cdot n^{4/p},
\end{equation*}
where $C$ is an absolute constant.

This completes the proof of the first part of Theorem \ref{thm:thin-shell-conj}. 
For the second part, namely for the statement that $\sigma_{K_p}^2 \gtrsim 1$ when $p\gtrsim n^2 \log n$, see the next subsection.

\begin{remark}
Note that \eqref{eq:identity1}, \eqref{eq:identity2} and Lemma \ref{lem:multiples of p-norm} imply the thin-shell conjecture when $p=2$ as well.
Although this is not interesting when $E = {\cal M}_n({\mathbb F})$, given that in those cases we know that $\overline{K}_{2,E}$ is
simply the Euclidean ball of volume 1 in $E$, 
it is perhaps worth noting in the case of Hermitian matrices, of anti-symmetric Hermitian, or of complex symmetric matrices
(especially so if ${\overline K}_{2,E}$ turns out not to be isotropic for one or more of these three subspaces $E$).
\end{remark}

\begin{remark}
Note that, since $\overline{K}_{p,E}$ is in isotropic position when $E = {\cal M}_n({\mathbb R})$ or ${\cal M}_n({\mathbb C})$, we have that
\begin{align*} 
d_nL_{K_{p,E}}^2 = \frac{1}{|K_{p,E}|^{1+\frac{2}{d_n}}}\int_{K_{p,E}} \|s(T)\|_2^2\,dT
&= \frac{1}{|K_{p,E}|^{2/d_n}} \left(\frac{1}{|K_{p,E}|}\int_{K_{p,E}} \|s(T)\|_2^2\,dT\right) 
\\
& = \frac{1}{|K_{p,E}|^{2/d_n}}\frac{\Gamma\left(1+\frac{d_n}{p}\right)}{\Gamma\left(1+\frac{d_n + 2}{p}\right)} \frac{M_p\bigl(\|x\|_2^2\bigr)}{M_p(1)}
\\
\intertext{which for large $p$ can be rewritten, using \eqref{eqp1:thm:thin-shell-conj}, as}
\\
& = \left(1 + O\left(\frac{\log n}{p}\right)\right) \frac{d_n}{|K_{p,E}|^{2/d_n}} \left(\frac{2d_n + (1-c)n}{n}\right)^{-1}.
\end{align*}
This shows that, for $p >>\log n$,
\begin{align*}L_{K_{p,E}} &=  \left(1 + O\left(\frac{\log n}{p}\right) + O\left(\frac{1}{n}\right)\right) \sqrt{\frac{n}{2d_n}}\frac{1}{|K_{p,E}|^{1/d_n}} 
\\
&= \left(1 + O\left(\frac{\log n}{p}\right) + O\left(\frac{1}{n}\right)\right) \frac{1}{\sqrt{2\beta n}} \frac{1}{|K_{p,E}|^{1/d_n}},
\end{align*}
where $\beta =1$ if $E = {\cal M}_n({\mathbb R})$, and $\beta =2$ if $E = {\cal M}_n({\mathbb C})$.

Recall now that Saint-Raymond \cite{SaintRaymond-1984} has found very precise estimates for the volume of $K_{p, {\cal M}_n({\mathbb R})}$
and of $K_{p{\cal M}_n({\mathbb C})}$, and in particular he has shown that
\begin{equation*} |K_{\infty, {\cal M}_n({\mathbb R})}|^{1/n^2} = (1+ o(1))\,\frac{1}{2}\sqrt{\frac{2\pi e^{3/2}}{n}},\quad \quad  
|K_{\infty, {\cal M}_n({\mathbb C})}|^{1/(2n^2)} = (1+ o(1))\,\frac{1}{2}\sqrt{\frac{\pi e^{3/2}}{n}}.
\end{equation*}
We can thus conclude that
\begin{equation*} L_{K_{\infty, {\cal M}_n({\mathbb R})}} = (1+ o(1))\frac{1}{\sqrt{\pi e^{3/2}}} = L_{K_{\infty, {\cal M}_n({\mathbb C})}}.\end{equation*}
\end{remark}

\bigskip

\subsection{Necessity of a negative correlation property}

Here we prove Theorem \ref{thm:neg-cor-prop}. Let us set
\begin{equation*} \frac{M_p\bigl(x_1^2\bigr)}{M_p(1)} = c_2\frac{M_p\bigl(|x_1|^p\bigr)}{M_p(1)} \end{equation*}
where $c_2$ depends both on $p$ and on $n$. Then by \eqref{eq:identity1} we can also write
\begin{equation*}
\frac{M_p\bigl(|x_1|^{p+2}\bigr)}{M_p(1)} = \frac{2d_n + (1 - c)n}{pn}\cdot c_2\frac{M_p\bigl(|x_1|^p\bigr)}{M_p(1)}. 
\end{equation*}
Using this and the fact that
\begin{equation*} M_p\left[|x_1|^p \left(x_1^2 - \frac{2d_n + (1-c)n}{pn} c_2\right)^2\right] \gs 0,\end{equation*}
we can deduce that
\begin{equation}\label{eqp1:thm:neg-cor-prop}
p\,\frac{M_p\bigl(|x_1|^{p+4}\bigr)}{M_p(1)} \gs \frac{\bigl(2d_n + (1-c) n\bigr)^2}{d_n\cdot n} \left(\frac{M_p\bigl(x_1^2\bigr)}{M_p(1)}\right)^2. 
\end{equation}
Furthermore, as we mentioned in Remark \ref{rem:Var Kp and Mp}, we have
\begin{equation*}{\rm Var}_{M_p}\bigl(\|x\|_2^2\bigr) \lesssim n\cdot n^{4/p} = o(n^2)\cdot n^{4/p}, \end{equation*}
which implies that
\begin{equation}\label{eqp2:thm:neg-cor-prop} \frac{M_p\bigl(x_1^2x_2^2\bigr)}{M_p(1)} = 
(1 + o(1))\left(\frac{M_p\bigl(x_1^2\bigr)}{M_p(1)}\right)^2.\end{equation}
Combining \eqref{eqp1:thm:neg-cor-prop}-\eqref{eqp2:thm:neg-cor-prop} with identity \eqref{eq:identity3} we obtain that
\begin{align*}
\left(\frac{2d_n + (3-c)n}{n}\right) \frac{M_p\bigl(x_1^4\bigr)}{M_p(1)} &= p \frac{M_p\bigl(|x_1|^{p+4}\bigr)}{M_p(1)} 
- \,\left(\frac{d_n - (c+1)n}{n}\right)\frac{M_p\bigl(x_1^2x_2^2\bigr)}{M_p(1)}
\\
& \gs \frac{\bigl(2d_n + (1-c) n\bigr)^2}{d_n\cdot n} \left(\frac{M_p\bigl(x_1^2\bigr)}{M_p(1)}\right)^2 - 
\,\frac{d_n}{n}\, (1+o(1))\left(\frac{M_p\bigl(x_1^2\bigr)}{M_p(1)}\right)^2.
\end{align*}
This gives inequality \eqref{eq1:thm:neg-cor-prop} of Theorem \ref{thm:neg-cor-prop}, and the rest of the theorem
follows as a consequence too.

\bigskip

It remains to estimate $\sigma_{K_\infty}$ more accurately and establish the second part of Theorem \ref{thm:thin-shell-conj} as well.
Combining identities \eqref{eq:identity2} and \eqref{eq:identity3}, we get
\begin{align*} 
\left(\frac{2d_n + (1 - b - c)n}{n}\right)\frac{M_p\bigl(x_1^2x_2^2\bigr)}{M_p(1)} 
&= p \frac{M_p\bigl(x_1^2|x_2|^{p+2}\bigr)}{M_p(1)}
\\ 
&= \frac{p}{n-1}\left(\frac{M_p\bigl(x_1^2\|x\|_{p+2}^{p+2}\bigr)}{M_p(1)} - \frac{M_p\bigl(|x_1|^{p+4}\bigr)}{M_p(1)}\right).  
\end{align*}
Provided that $p$ is large enough, we can compute the latter terms with great accuracy:
\begin{align*} \frac{M_p\bigl(|x_1|^{p+4}\bigr)}{M_p(1)} &= \frac{1}{n} \frac{M_p\bigl(\|x\|_{p+4}^{p+4}\bigr)}{M_p(1)}
\\
& = \left(\frac{1}{n}+ O\left(\frac{\log n}{pn}\right)\right)\frac{M_p\bigl(\|x\|_p^{p+4}\bigr)}{M_p(1)}
\\
& = \left(\frac{1}{n}+ O\left(\frac{\log n}{pn}\right)\right)\frac{\Gamma\left(\frac{d_n+p+4}{p}\right)}{\Gamma\left(\frac{d_n}{p}\right)},
\end{align*}
while similarly
\begin{align*}
\frac{M_p\bigl(x_1^2\|x\|_{p+2}^{p+2}\bigr)}{M_p(1)} & = 
\left(1+ O\left(\frac{\log n}{p}\right)\right) \frac{M_p\bigl(x_1^2\|x\|_p^{p+2}\bigr)}{M_p(1)}
\\
& = \left(1+ O\left(\frac{\log n}{p}\right)\right)\frac{\Gamma\left(\frac{d_n+p+4}{p}\right)}{\Gamma\left(\frac{d_n+2}{p}\right)} 
\frac{M_p\bigl(x_1^2\bigr)}{M_p(1)}.
\end{align*}
Using also identity \eqref{eq:identity1} now, we can continue by writing
\begin{align*} 
\left(\frac{2d_n + (1 - c)n}{n}\right)\frac{M_p\bigl(x_1^2\bigr)}{M_p(1)} &= p \frac{M_p\bigl(|x_1|^{p+2}\bigr)}{M_p(1)}
\\
& = \left(1+ O\left(\frac{\log n}{p}\right)\right)\cdot \frac{d_n + 2}{n}\frac{\Gamma\left(\frac{d_n+2}{p}\right)}{\Gamma\left(\frac{d_n}{p}\right)}.
\end{align*}
In the end,
\begin{equation*} 
\left(\frac{2d_n + (1 - b - c)n}{n}\right)\frac{M_p\bigl(x_1^2x_2^2\bigr)}{M_p(1)} 
= \frac{p}{n-1}\left(1+ O\left(\frac{\log n}{p}\right)\right) \frac{\Gamma\left(\frac{d_n+p+4}{p}\right)}{\Gamma\left(\frac{d_n}{p}\right)}
\left(\frac{d_n + 2}{2d_n + (1-c)n} -\frac{1}{n}\right),
\end{equation*}
or in other words
\begin{gather}
\frac{M_p\bigl(x_1^2x_2^2\bigr)}{M_p(1)} 
=  \left(1+ O\left(\frac{\log n}{p}\right)\right) 
 \frac{\Gamma\left(1+\frac{d_n+4}{p}\right)}{\Gamma\left(1+\frac{d_n}{p}\right)} 
 \frac{d_n\bigl(nd_n - 2d_n + (1+c)n\bigr)}{(n-1)\bigl(2d_n + (1-c)n\bigr)\bigl(2d_n + (1-b-c)n\bigr)},
\\
\intertext{while}
\frac{M_p\bigl(x_1^2\bigr)}{M_p(1)} 
 = \left(1+ O\left(\frac{\log n}{p}\right)\right) \frac{\Gamma\left(1+\frac{d_n+2}{p}\right)}{\Gamma\left(1+\frac{d_n}{p}\right)} \frac{d_n}{2d_n + (1-c) n}.
\end{gather}

Returning to \eqref{eq:identity3}
we also see that
\begin{gather*}
\left(\frac{2d_n + (3 - c)n}{n}\right)\frac{M_p\bigl(x_1^4\bigr)}{M_p(1)}  
= p\frac{M_p\bigl(|x_1|^{p+4}\bigr)}{M_p(1)} - b(n-1) \frac{M_p\bigl(x_1^2x_2^2\bigr)}{M_p(1)}
\\
 = \left(1+ O\left(\frac{\log n}{p}\right)\right)\frac{\Gamma\left(1+\frac{d_n+4}{p}\right)}{\Gamma\left(1+\frac{d_n}{p}\right)}
\left[\frac{d_n}{n} -  \frac{bd_n\bigl(nd_n - 2d_n + (1+c)n\bigr)}{\bigl(2d_n + (1-c)n\bigr)\bigl(2d_n + (1-b-c)n\bigr)} \right],
\end{gather*}
or in other words
\begin{multline}
\frac{M_p\bigl(x_1^4\bigr)}{M_p(1)}  =
\left(1+ O\left(\frac{\log n}{p}\right)\right)\frac{\Gamma\left(1+\frac{d_n+4}{p}\right)}{\Gamma\left(1+\frac{d_n}{p}\right)}
\left[\frac{d_n}{2d_n + (3-c)n} \ -\right.
\\
\left. - \  \frac{b\,nd_n\bigl(nd_n - 2d_n + (1+c)n\bigr)}{\bigl(2d_n + (1-c)n\bigr)\bigl(2d_n + (1-b-c)n\bigr)\bigl(2d_n + (3-c)n\bigr)} \right].
\end{multline}

Combining all these, and taking into account that if $p\gs n^2 \log n$ then $d_n/p = o(1)$, as well as
Lemma \ref{lem:Gamma estimates}(b), we conclude that, for all such $p$,
\begin{gather*} 
n\cdot  \frac{M_p\bigl(x_1^4\bigr)}{M_p(1)} + n(n-1) \cdot \frac{M_p\bigl(x_1^2x_2^2\bigr)}{M_p(1)} - 
n^2\cdot \left(\frac{M_p\bigl(x_1^2\bigr)}{M_p(1)}\right)^2 =
\\
\bigl(1+ o(1)\bigr) \cdot 
\left[n\cdot \left(\frac{d_n}{2d_n + (3-c)n} - \frac{b\,nd_n\bigl(nd_n - 2d_n + (1+c)n\bigr)}{\bigl(2d_n + (1-c)n\bigr)\bigl(2d_n + (1-b-c)n\bigr)\bigl(2d_n + (3-c)n\bigr)}\right)\  + \right.
\\
\left. +\  n(n-1)\cdot \frac{d_n\bigl(nd_n - 2d_n + (1+c)n\bigr)}{(n-1)\bigl(2d_n + (1-c)n\bigr)\bigl(2d_n + (1-b-c)n\bigr)} 
- n^2\cdot \left(\frac{d_n}{2d_n + (1-c) n}\right)^2 \right] 
\\
=\, \frac{1}{8b} + o(1).
\end{gather*}
Given in addition that 
\begin{equation*}{\rm Var}_{M_p}\bigl(\|x\|_2^2\bigr) \simeq \max\Bigl\{\sigma_{K_p}^2, \frac{1}{p}\Bigr\}\cdot n^{4/p},\end{equation*}
and that for all $d_n$-dimensional bodies $K$ we have $\sigma_K^2 \gtrsim 1/d_n$, we obtain that
\begin{equation*} \sigma_{K_p}^2 \simeq {\rm Var}_{M_p}\bigl(\|x\|_2^2\bigr) \simeq \frac{1}{8b} + o(1)\end{equation*}
for all $p\gtrsim n^2\log n$.

\medskip

\section{The cases of complex anti-symmetric and of Hermitian matrices} \label{sec:Special cases}

We will now show why we have analogues of Proposition \ref{prop:Var Kp and Mp}, 
Theorem \ref{thm:thin-shell-conj} and Theorem \ref{thm:neg-cor-prop}
for the Schatten classes in the subspaces of anti-symmetric Hermitian matrices and of Hermitian matrices too. 
We stated in Section \ref{sec:Notation}, Fact 4, that, if the subspace of anti-symmetric Hermitian matrices is equipped
with the standard Gaussian measure, then the induced joint probability density of the singular values 
$\theta_1,\ldots, \theta_s$ of a random matrix $T\in E$ is given by
\begin{gather*} 
{\mathbb P}_n\bigl((\theta_1,\ldots,\theta_s)\in A\bigr)
= c_{n,E} \cdot \int_A \prod_{1\ls i<j\ls s}\big|x_i^2 - x_j^2\big|^2\, \exp\bigl(-\|x\|_2^2\bigr)\,dx
\\
\intertext{if $n = 2s$, and by}
{\mathbb P}_n\bigl((\theta_1,\ldots,\theta_s)\in A\bigr) 
= c_{n,E} \cdot \int_A \prod_{1\ls i<j\ls s}\big|x_i^2 - x_j^2\big|^2 \cdot 
\prod_{1\ls i\ls s} |x_i|^2\, \exp\bigl(-\|x\|_2^2\bigr)\,dx
\end{gather*}
if $n = 2s +1$, where $A$ is any $1$-symmetric measurable subset of ${\mathbb R}^s$. This of course implies 
that for every symmetric measurable function $F:{\mathbb R}^s \to {\mathbb R}^+$ we must have
\begin{multline} \label{eq:anti-symmetric reduction to n-integrals}
\int_E F(\theta_1,\ldots,\theta_s) \exp\bigl(-\|T\|_{S_2^n}^2/2\bigr)\, dT = 
\\
c_{n,E} \cdot \int_{{\mathbb R}^s} F\bigl(|x_1|,\ldots,|x_s|\bigr)\cdot 
\prod_{1\ls i<j\ls s} \big|x_i^2 - x_j^2\big|^2\cdot \prod_{1\ls i\ls s} |x_i|^{2r}\, \exp\bigl(-\|x\|_2^2\bigr)\,dx,
\end{multline}
where $n=2s + r$, $r\in \{0,1\}$. This allows us to prove the following

\begin{lemma}
Let $F:{\mathbb R}^s \to {\mathbb R}^+$ be a measurable, symmetric and $k$-homogeneous function. 
Then
\begin{multline*} \int_{K_{p,E}} F(\theta_1,\ldots,\theta_s)\,dT = 
\\
\frac{c_{n,E}\,2^{-(n^2+k)/p}}{\Gamma\!\left(1+\frac{n^2+k}{p}\right)} \int_{{\mathbb R}^s} F\bigl(|x_1|,\ldots,|x_s|\bigr)\cdot 
\prod_{1\ls i<j\ls s} \big|x_i^2 - x_j^2\big|^2\cdot \prod_{1\ls i\ls s} |x_i|^{2r}\, \exp\bigl(-\|x\|_p^p\bigr)\,dx,
\end{multline*}
where $n^2 = {\rm dim}(E)$.
\end{lemma}
\begin{proof}
We first note that, for every anti-symmetric Hermitian matrix $T$ and for each $p\gs 1$,
\begin{equation*} \|T\|_{S_p^n}^p = 2\sum_{i=1}^s |\theta_i|^p = 2\|(\theta_1,\ldots,\theta_s)\|_p^p,\end{equation*}
therefore, applying \eqref{eq:anti-symmetric reduction to n-integrals} with the function
\begin{equation*} F(x_1,\ldots, x_s)\cdot \frac{\exp(-\|x\|_p^p)}{\exp(-\|x\|_2^2)},\end{equation*}
we see that
\begin{multline*} \int_E F(\theta_1,\ldots,\theta_s) \exp\Bigl(-\|T\|_{S_p^n}^p/2\Bigr)\, dT = 
\\
c_{n,E} \cdot \int_{{\mathbb R}^s} F\bigl(|x_1|,\ldots,|x_s|\bigr)\cdot 
\prod_{1\ls i<j\ls s} \big|x_i^2 - x_j^2\big|^2\cdot \prod_{1\ls i\ls s} |x_i|^{2r}\, \exp\bigl(-\|x\|_p^p\bigr)\,dx.\end{multline*}
Furthermore, given that $K_{p,E} = \bigl\{T\in E : \|T\|_{S_p^n}\ls 1\bigr\}$,
we can write
\begin{align*}
\int_E F(\theta_1,\ldots,\theta_s) \exp\Bigl(-\|T\|_{S_p^n}^p/2\Bigr)\, dT 
& = \int_0^{+\infty}e^{-t}\int_{(2t)^{1/p} K_{p,E}} F(\theta_1,\ldots,\theta_s)\,dT\, dt
\\
& = 2^{(n^2+k)/p}\ \Gamma\!\left(1+\frac{n^2+k}{p}\right)\cdot \int_{K_{p,E}} F(\theta_1,\ldots,\theta_s)\,dT.
\end{align*}
This concludes the proof.
\end{proof}

Note now that, if $M_{2,2,2r,p}$ denotes integration on ${\mathbb R}^s$ with respect to the density
\begin{equation*} \exp\bigl(-\|x\|_p^p\bigr)\cdot f_{2,2,2r}(x) =\exp\bigl(-\|x\|_p^p\bigr)\cdot \prod_{1\ls i<j\ls s}\big|x_i^2 - x_j^2\big|^2\cdot \prod_{1\ls i\ls s}|x_i|^{2r},\end{equation*}
where $r\in\{0,1\}$, then 
\begin{align*} {\mathbb E}_{\overline{K}_{p,E}}\bigl(\|s(T)\|_2^2\bigr) 
& = \frac{1}{|K_{p,E}|^{2/d_n}}\frac{\int_{K_{p,E}} 2\|(\theta_1,\ldots,\theta_s\|_2^2\,dT}{|K_{p,E}|}
\\
& = \frac{1}{|K_{p,E}|^{2/d_n}} \cdot \frac{\Gamma\!\left(1+\frac{n^2}{p}\right)}{2^{2/p}\,\Gamma\!\left(1+\frac{n^2+2}{p}\right)} \,
\frac{M_{2,2,2r,p}\bigl(\|x\|_2^2\bigr)}{M_{2,2,2r,p}(1)},
\end{align*}
and similarly
\begin{gather*}
|K_{p,E}|^{4/d_n}\cdot {\rm Var}_{\overline{K}_{p,E}}\bigl(\|s(T)\|_2^2\bigr) 
 = \frac{1}{|K_{p,E}|} \int_{K_{p,E}} \|s(T)\|_2^4\,dT - \left(\frac{1}{|K_{p,E}|}\int_{K_{p,E}} \|s(T)\|_2^2\,dT\right)^2
\\
 = \frac{1}{2^{4/p}}\, \left[\frac{\Gamma\left(1+\frac{n^2}{p}\right)}{\Gamma\left(1+\frac{n^2+4}{p}\right)}\,
\frac{M_{2,2,2r,p}\bigl(\|x\|_2^4\bigr)}{M_{2,2,2r,p}(1)} - 
\left(\frac{\Gamma\left(1+\frac{n^2}{p}\right)}{\Gamma\left(1+\frac{n^2+2}{p}\right)}\right)^2\,
\left(\frac{M_{2,2,2r,p}\bigl(\|x\|_2^2\bigr)}{M_{2,2,2r,p}(1)}\right)^2\right].
\end{gather*}

From this point on, we can proceed as in Sections \ref{sec:reduction-to-Mp} and \ref{sec:main-proofs} to prove that
\begin{equation*} {\mathbb E}_{\overline{K}_{p,E}}\bigl(\|s(T)\|_2^2\bigr) \simeq n^{1-2/p} \cdot 
s\cdot \frac{M_{2,2,2r,p}\bigl(x_1^2\bigr)}{M_{2,2,2r,p}(1)} \simeq n^2 = {\rm dim}(E),
\end{equation*}
as well as Theorems \ref{thm:thin-shell-conj} and \ref{thm:neg-cor-prop} 
(note that this time, when we apply Lemma \ref{lem:int-by-parts}
and Propositions \ref{prop:order-of-magnitudes} and \ref{prop:int-by-parts-&-sym}, $d_n$ is replaced by $d_s = 2s(s-1) + (2r + 1)s$,
which is not equal to ${\rm dim}(E)$, still the conclusions we obtain are of the same form given that $s=\lfloor \frac{n}{2}\rfloor \simeq n$).

\bigskip

Let us finally see why Theorems \ref{thm:thin-shell-conj} and \ref{thm:neg-cor-prop} hold true when $E$ is the subspace of Hermitian matrices,
even though by Fact 2 
we know that the density we have to use when we reduce integrals over the balls $K_{p,E}$ to integrals over
${\mathbb R}^n$ is the density $f_{1,2,0}(x)$. 

\begin{proposition}
Let $f:{\mathbb R}^n\to {\mathbb R}^+$ be a symmetric function. Then there exists a constant $c_n$ depending only on $n$ such that
\begin{multline*} \int_{{\mathbb R}^n} f\bigl(|x_1|,\ldots,|x_n|\bigr)\cdot f_{1,2,0}(x)\, e^{-\|x\|_p^p}\, dx = 
\\
\sum_{\substack{A\subset [n]\\ |A| = \lceil n/2\rceil}} c_n\cdot  \int_{{\mathbb R}^n} f\bigl(|x_1|,\ldots,|x_n|\bigr)\, e^{-\|x\|_p^p}\cdot
\prod_{i,j\in A; i<j} \big|x_i^2 - x_j^2\big|^2\cdot \prod_{l,k\notin A; l<k} \big|x_l^2 - x_k^2\big|^2\cdot \prod_{l\notin A} |x_l|^2\, dx.
 \end{multline*}
\end{proposition}
\begin{proof}
Since the integrand $f\bigl(|x_1|,\ldots,|x_n|\bigr)\, e^{-\|x\|_p^p}$ is invariant under permutations of the coordinates or flipping
of their signs, we can make use of the exact argument of Edelman and La Croix from \cite[Section 4]{Edelman-LaCroix-2015}
to obtain the result.
\end{proof}

We now remark that, with $f(x) = \|x\|_\xi^\xi$ or $f(x) = {\bf 1}$, we have
\begin{multline*} 
\sum_{\substack{A\subset [n]\\ |A| = \lceil n/2\rceil}} c_n\cdot \int_{{\mathbb R}^n} f(x)\, e^{-\|x\|_p^p}\cdot
\prod_{i,j\in A; i<j} \big|x_i^2 - x_j^2\big|^2\cdot \prod_{l,k\notin A; l<k} \big|x_l^2 - x_k^2\big|^2\cdot \prod_{l\notin A} |x_l|^2\, dx
\\
= {n\choose {\lceil n/2\rceil}} \, c_n\cdot \int_{{\mathbb R}^n} f(x)\, e^{-\|x\|_p^p}\cdot
\prod_{i,j\in I_1; i<j} \big|x_i^2 - x_j^2\big|^2\cdot \prod_{l,k\notin I_1; l<k} \big|x_l^2 - x_k^2\big|^2\cdot \prod_{l\notin I_1} |x_l|^2\, dx
\end{multline*}
where $I_1$ is the subset of the first $n_1 = \lceil n/2\rceil$ indices from $\{1,2,\ldots, n\}$,
and where we will write $I_2$ for its complement. Let us denote by $N_{p,I_1}$ integration over ${\mathbb R}^{I_1}$
with respect to the density $\prod_{i,j\in I_1; i<j} \big|x_i^2 - x_j^2\big|^2\, e^{-\|x\|_{p,I_1}^p}$, 
where $\|x\|_{p,I_1}$ denotes the $p$-norm of the coordinates of $x$ with indices in $I_1$ only, 
and let us denote by $N_{p,I_2}$ integration over ${\mathbb R}^{I_2}$
with respect to the density $\prod_{l,k\in I_2; l<k} \big|x_l^2 - x_k^2\big|^2\cdot \prod_{l\in I_2} |x_l|^2\, e^{-\|x\|_{p,I_2}^p}$. 
Let us finally denote by $N_{p,I_1,I_2}$ integration over ${\mathbb R}^n$ with respect to the product of both densities. 
Then, by the independent nature of these two densities and by the above relations, we see that
\begin{equation*} \frac{M_{1,2,0,p}\bigl(\|x\|_\xi^\xi\bigr)}{M_{1,2,0,p}(1)} = \frac{N_{p,I_1,I_2}\bigl(\|x\|_\xi^\xi\bigr)}{N_{p,I_1,I_2}(1)}
= \frac{N_{p,I_1}\bigl(\|x\|_{\xi,I_1}^\xi\bigr)}{N_{p,I_1}(1)} + \frac{N_{p,I_2}\bigl(\|x\|_{\xi,I_2}^\xi\bigr)}{N_{p,I_2}(1)}.
\end{equation*}
Similarly we have that 
\begin{align*} 
\frac{M_{1,2,0,p}\bigl(\|x\|_2^4\bigr)}{M_{1,2,0,p}(1)} &= \frac{N_{p,I_1,I_2}\bigl(\|x\|_2^4\bigr)}{N_{p,I_1,I_2}(1)}
\\
&=  \frac{N_{p,I_1,I_2}\left(\|x\|_{2,I_1}^4 + \|x\|_{2,I_2}^4 + 2\|x\|_{2,I_1}^2 \|x\|_{2,I_2}^2\right)}{N_{p,I_1,I_2}(1)}
\\
& = \frac{N_{p,I_1}\bigl(\|x\|_{2,I_1}^4\bigr)}{N_{p,I_1}(1)} + \frac{N_{p,I_2}\bigl(\|x\|_{2,I_2}^4\bigr)}{N_{p,I_2}(1)} 
+ 2 \frac{N_{p,I_1}\bigl(\|x\|_{2,I_1}^2\bigr)}{N_{p,I_1}(1)} \frac{N_{p,I_2}\bigl(\|x\|_{2,I_2}^2\bigr)}{N_{p,I_2}(1)}. 
\end{align*}
Therefore, to show that 
\begin{equation*} \frac{M_{1,2,0,p}\bigl(\|x\|_2^4\bigr)}{M_{1,2,0,p}(1)} 
= \left(1 + O\left(\frac{1}{n^2}\right)\right) \left(\frac{M_{1,2,0,p}\bigl(\|x\|_2^2\bigr)}{M_{1,2,0,p}(1)}\right)^2 \end{equation*}
for some index $p$, we only need to establish that
\begin{align*} 
\frac{N_{p,I_1}\bigl(\|x\|_{2,I_1}^4\bigr)}{N_{p,I_1}(1)} 
&= \left(1 + O\left(\frac{1}{n^2}\right)\right)\left(\frac{N_{p,I_1}\bigl(\|x\|_{2,I_1}^2\bigr)}{N_{p,I_1}(1)}\right)^2
\\
\hbox{and that}\qquad \frac{N_{p,I_2}\bigl(\|x\|_{2,I_2}^4\bigr)}{N_{p,I_2}(1)} 
&= \left(1 + O\left(\frac{1}{n^2}\right)\right)\left(\frac{N_{p,I_2}\bigl(\|x\|_{2,I_2}^2\bigr)}{N_{p,I_2}(1)}\right)^2.
\end{align*}
But we have already seen the latter are true when $p\gs n^2\log n$ or when $p = 2$
(since $N_{p,I_1}$ is exactly $M_{2,2,0,p}$ over ${\mathbb R}^{I_1}$, while $N_{p,I_2}$ 
stands for $M_{2,2,2,p}$ over ${\mathbb R}^{I_2}$). 

Similarly for those $p$ we obtain that
\begin{align*} 
\frac{N_{p,I_1}\bigl(x_1^4\bigr)}{N_{p,I_1}(1)} 
&\gs \left(\frac{3}{2} + o(1)\right)\left(\frac{N_{p,I_1}\bigl(x_1^2\bigr)}{N_{p,I_1}(1)}\right)^2
\\
\hbox{and}\qquad \frac{N_{p,I_2}\bigl(x_n^4\bigr)}{N_{p,I_2}(1)} 
&\gs \left(\frac{3}{2} + o(1)\right)\left(\frac{N_{p,I_2}\bigl(x_n^2\bigr)}{N_{p,I_2}(1)}\right)^2,
\end{align*}
whence inequality \eqref{eq1:thm:neg-cor-prop} of Theorem \ref{thm:neg-cor-prop} follows:
\begin{align*}
\frac{M_{1,2,0,p}\bigl(\|x\|_4^4\bigr)}{M_{1,2,0,p}(1)} 
& = \frac{N_{p,I_1}\bigl(\|x\|_{4,I_1}^4\bigr)}{N_{p,I_1}(1)} + \frac{N_{p,I_2}\bigl(\|x\|_{4,I_2}^4\bigr)}{N_{p,I_2}(1)}
\\
&\gs \left(\frac{3}{2} + o(1)\right)\left[\frac{1}{n_1}\left(\frac{N_{p,I_1}\bigl(\|x\|_{2,I_1}^2\bigr)}{N_{p,I_1}(1)}\right)^2
+ \frac{1}{n_2}\left(\frac{N_{p,I_2}\bigl(\|x\|_{2,I_2}^2\bigr)}{N_{p,I_2}(1)}\right)^2\right]
\\
&\gs \left(\frac{3}{2} + o(1)\right)\,\frac{1}{n} 
\left(\frac{N_{p,I_1}\bigl(\|x\|_{2,I_1}^2\bigr)}{N_{p,I_1}(1)} + \frac{N_{p,I_2}\bigl(\|x\|_{2,I_2}^2\bigr)}{N_{p,I_2}(1)}\right)^2
\\
& = \left(\frac{3}{2} + o(1)\right)\,\frac{1}{n} \left(\frac{M_{1,2,0,p}\bigl(\|x\|_2^2\bigr)}{M_{1,2,0,p}(1)}\right)^2.
\end{align*}

\begin{remark}\label{rem:conditional-estimates}
Note that for the subspaces of anti-symmetric Hermitian and of Hermitian matrices, as well as for the subspace of complex symmetric spaces,
inequality \eqref{eq1:thm:neg-cor-prop} holds only conditionally, depending on whether we have
\begin{equation*} \frac{M_p\bigl(x_1^2 x_2^2\bigr)}{M_p(1)} = \bigl(1 + o(1)\bigr)\, \left(\frac{M_p\bigl(x_1^2\bigr)}{M_p(1)}\right)^2,
\end{equation*}
or equivalently whether we have $\sigma_{K_{p,E}}^2 = o(n^2)$, which we a priori do not know for these balls $K_{p,E}$. 
Nevertheless, since through the arguments for Theorem \ref{thm:thin-shell-conj} we can conclude that $\sigma_{K_{p,E}}^2 = o(n^2)$
for all $p>>\log n$ for these subspaces too, inequality \eqref{eq1:thm:neg-cor-prop} holds unconditionally in this range of $p$.

At any rate, the final conclusion of Theorem \ref{thm:neg-cor-prop} remains unaffected: given any $p\gs 1$, for the thin-shell conjecture
to hold true for $K_{p,E}$, where $E$ is any of the three subspaces mentioned here, or even for $\sigma_{K_{p,E}}^2$ to be $o(n)$,
we need the density $f_{a,b,c,p}$ corresponding to $K_{p,E}$ to possess a negative correlation property.
\end{remark}

\smallskip

\section{More on the negative correlation property when $E = {\cal M}_n({\mathbb F})$}\label{sec:neg-cor-prop-revisited}

The purpose of this final section is to establish a type of negative correlation property for the original, uniform measures on 
$K_{p,{\cal M}_n({\mathbb F})}$ as well. We start with the following lemma which allows us to relate terms 
that appear when we expand ${\rm Var}_{M_p}\bigl(\|x\|_2^2\bigr)$ and ${\rm Var}_{\overline{K}_{p,E}}\bigl(\|T\|_{HS}^2\bigr)$
respectively.

\begin{lemma}\label{lem:sing-val-and-entries} 
For every $n\times n$ matrix $T\in {\cal M}_n({\mathbb F})$, where ${\mathbb F} = {\mathbb R}$ or ${\mathbb C}$
or ${\mathbb H}$, we have that, if $T = (a_{i,j})_{1\ls i, j\ls n}$ and if $(s_i(T))_{1\ls i\ls n} = (s_i)_{1\ls i\ls n}$ is the non-increasing 
rearrangement of the singular values of $T$, then
\begin{equation}\label{eq1:lem:sing-val-and-entries}
\sum_{i=1}^n s_i^4 = \sum_{1\ls i,j\ls n} |a_{i,j}|^4 + \sum_{i=1}^n\sum_{j\neq l}\bigl(|a_{i,j}|^2|a_{i,l}|^2 + |a_{j,i}|^2|a_{l,i}|^2\bigr)
+ \sum_{i\neq l}\sum_{j\neq k} a_{i,j}\overline{a_{l,j}}a_{l,k}\overline{a_{i,k}}, 
\end{equation}
while
\begin{equation} \label{eq2:lem:sing-val-and-entries}\sum_{i\neq j} s_i^2s_j^2  = \sum_{i\neq l}\sum_{j\neq k} |a_{i,j}|^2|a_{l,k}|^2 - 
\sum_{i\neq l}\sum_{j\neq k} a_{i,j}\overline{a_{l,j}}a_{l,k}\overline{a_{i,k}}.\end{equation}
\end{lemma}
\begin{remark}
When the entries of the matrix $T$ are real or complex numbers, we have that multiplication between different entries is commutative, 
hence we can rewrite \eqref{eq2:lem:sing-val-and-entries} as
\begin{align} \label{eq3:lem:sing-val-and-entries}
\sum_{i\neq j} s_i^2s_j^2 & = \sum_{i\neq l}\sum_{j\neq k} |a_{i,j}|^2|a_{l,k}|^2 - \sum_{i\neq l}\sum_{j\neq k} a_{i,j}a_{l,k}\overline{a_{l,j}a_{i,k}}
\\
\nonumber
& = 2\sum_{i<l}\sum_{j\neq k} a_{i,j}a_{l,k}\bigl(\overline{a_{i,j}a_{l,k}} - \overline{a_{i,k}a_{l,j}}\bigr)
 = 2\sum_{i<l}\sum_{j < k} \big|a_{i,j}a_{l,k} - a_{i,k}a_{l,j}\big|^2.
\end{align}
This is of course not necessarily true when $T\in {\cal M}_n({\mathbb H})$, since ${\mathbb H}$ is a skew field.
Note however that the last sum in both \eqref{eq1:lem:sing-val-and-entries} and \eqref{eq2:lem:sing-val-and-entries}
is a real number in all cases.
\end{remark}
\begin{proof}
Note that $\sum_{i=1}^n s_i^4 = {\rm tr}\bigl((T^\ast T)^2\bigr) = {\rm tr}\bigl((T T^\ast)^2\bigr)$. We also have that 
\begin{equation*} TT^\ast = \left(\sum_{j=1}^n a_{i,j}\overline{a_{l,j}}\right)_{1\ls i, l\ls n},\end{equation*}
thus the $(i,i)$-th entry of $(T T^\ast)^2$ is equal to
\begin{equation*}
\bigl((T T^\ast)^2\bigr)_{i,i} = \sum_{l=1}^n\left(\sum_{j=1}^na_{i,j}\overline{a_{l,j}}\right)\left(\sum_{k=1}^na_{l,k}\overline{a_{i,k}}\right)
= \sum_{l=1}^n\sum_{1\ls j, k\ls n} a_{i,j}\overline{a_{l,j}}a_{l,k}\overline{a_{i,k}}.
\end{equation*}
Summing over all $i\in \{1,\ldots, n\}$ we get \eqref{eq1:lem:sing-val-and-entries}.

To also obtain \eqref{eq2:lem:sing-val-and-entries}, we recall that $\sum_{i=1}^n s_i^2 = \|T\|_{HS}^2 = \sum_{1\ls i,j\ls n} |a_{i,j}|^2$,
and also that $\left(\sum_{i=1}^n s_i^2\right)^2 = \sum_{i=1}^n s_i^4 + \sum_{i\neq j}s_i^2s_j^2$.
Thus
\begin{align*}
\sum_{i=1}^n s_i^4 + \sum_{i\neq j}s_i^2s_j^2 & = \left(\sum_{1\ls i,j\ls n} |a_{i,j}|^2\right)^2
\\
& = \sum_{1\ls i,j\ls n} |a_{i,j}|^4 + \sum_{i=1}^n\sum_{j\neq l}\bigl(|a_{i,j}|^2|a_{i,l}|^2 + |a_{j,i}|^2|a_{l,i}|^2\bigr)
+ \sum_{i\neq l}\sum_{j\neq k} |a_{i,j}|^2|a_{l,k}|^2,
\end{align*}
which combined with \eqref{eq1:lem:sing-val-and-entries} leads to \eqref{eq2:lem:sing-val-and-entries}.
\end{proof}

\medskip

We now need to study the orders of magnitude of the terms appearing in
\eqref{eq1:lem:sing-val-and-entries}-\eqref{eq2:lem:sing-val-and-entries}. 
This will be done through the study of symmetries of the balls $K_{p,{\cal M}_n({\mathbb F})}$, 
one immediate consequence of which is the isotropicity of these convex bodies.

\begin{lemma}\label{lem:permute-matrix}
Suppose $p\gs 1$ and let $E = {\cal M}_n({\mathbb R})$ or ${\cal M}_n({\mathbb C})$
or ${\cal M}_n({\mathbb H})$. If $A: E \to E$ is an invertible transformation that can be realised as left or right multiplication 
by an orthogonal or unitary or symplectic matrix respectively, then $A(K_{p,E}) = K_{p,E}$.
The same conclusion is true if $A$ takes a matrix in $E$ to its conjugate transpose, or simply to its transpose.
Immediate consequences are the following:
\begin{enumerate}
\item
For every $p\gs 1$, the normalised unit ball $\overline{K}_{p,E}$ is in isotropic position.
\item
For all $i, j\in \{1,2,\ldots, n\}$, and for every power $s>0$,
\begin{equation*}\int_{K_{p,E}} |a_{i,j}|^s\,dT  = \int_{K_{p,E}} |a_{1,1}|^s\,dT.\end{equation*}
\item
For all $i, j, l, k\in \{1,2,\ldots, n\}$ with $i\neq l$, $j\neq k$,
\begin{align*}
\int_{K_{p,E}} |a_{i,j}|^2|a_{i,k}|^2\,dT = \int_{K_{p,E}} |a_{j,i}|^2|a_{k,i}|^2\,dT  &= \int_{K_{p,E}} |a_{1,1}|^2|a_{1,2}|^2\,dT  = \int_{K_{p,E}} |a_{1,1}|^2|a_{2,1}|^2\,dT
\\ \\
\int_{K_{p,E}} |a_{i,j}|^2|a_{l,k}|^2\,dT  &= \int_{K_{p,E}} |a_{1,1}|^2|a_{2,2}|^2\,dT
\\
\int_{K_{p,E}} a_{i,j}\overline{a_{l,j}}a_{l,k}\overline{a_{i,k}}\,dT &= \int_{K_{p,E}} a_{1,1}\overline{a_{2,1}}a_{2,2}\overline{a_{1,2}}\,dT.
\end{align*}
\end{enumerate}
\end{lemma}
\begin{proof}
Let $U$ be an orthogonal (or unitary) matrix. Then for every $n\times n$ matrix $M$, we have that the singular values of 
$UM$ or of $MU$ are the same as those of $M$: indeed, $(UM)^\ast (UM) = M^\ast (U^\ast U)M = M^\ast M$, 
while $(MU)^\ast (MU) = U^\ast (M^\ast M) U$, so it has the same eigenvalues as $M^\ast M$. This implies that
$\{UM : M \in K_{p,E}\}$ or $\{MU : M \in K_{p,E}\}$ coincide with $K_{p,E}$.

On the other hand, if $A(M) = M^\ast$, then $\bigl(A(M)^\ast A(M)\bigr) = M M^\ast$, which has the same eigenvalues as $M^\ast M$.
The latter is true even if $A(M)$ is just the transpose of $M$.

Finally, if $A$ is a linear transformation on $E={\cal M}_n({\mathbb F})$ of one of the above forms,
then, since $A(K_{p,E}) = K_{p,E}$, we must have that $|{\rm det}(A)| = 1$. This shows that for every integrable function $F$ on $K_{p,E}$,
\begin{equation}\label{eqp1:lem-permute-matrix} 
\int_{K_{p,E}} F(T) \, dT = \int_{A(K_{p,E})} F(T) \, dT = \int_{K_{p, E}} |{\rm det}(A)| \cdot F(A(T))\, dT = \int_{K_{p, E}} F(A(T))\, dT.\end{equation}
It is now easy to establish statements 2 and 3 of the lemma: we apply \eqref{eqp1:lem-permute-matrix} with $F$ being suitable functions
of the entries of $T\in K_{p,E}$, and the linear transformation $A$ being either multiplication from the left or from the right (or from both sides)
by a permutation matrix $P_{i,j}$ (formed by permuting the $i$-th and the $j$-th row of the identity matrix, and leaving all other rows unchanged),
or $A$ being the transformation that sends each matrix $T$ to its (conjugate) transpose.

To also show that $\overline{K}_{p,E}$ is isotropic, we need to prove in addition that all integrals of products of pairs of different entries
(or of pairs of real and imaginary parts of them) are 0. In the real case, all such integrals must be equal 
\begin{equation}\label{eqp2:lem-permute-matrix}  \hbox{either to}\ \ \int_{K_{p,E}} a_{1,1}a_{1,2} \, dT 
= \int_{K_{p,E}} a_{1,1}a_{2,1} \, dT,  
\qquad \hbox{or to}\ \  \int_{K_{p,E}} a_{1,1}a_{2,2}\,dT,\end{equation}
so we just have to show that the latter integrals are 0. For the first one, consider the rotation matrix
\begin{equation} \label{eqp3:lem-permute-matrix} 
U = \left(\begin{array}{cc} 
{\begin{array}{cc} \cos(\theta) & \sin(\theta) \\ -\sin(\theta) & \cos(\theta) \end{array}} & \hbox{\Large ${\bm 0}$}
\\ \\
\hbox{\Large ${\bm 0}$} &\hbox{\Large ${\rm Id}_{n-2}$}
\\ &
\end{array}\right),\end{equation}
and apply \eqref{eqp1:lem-permute-matrix} with $A$ being multiplication from the left by $U$ and $F$ being the absolute value 
of the first entry squared (or simply the first entry squared): since
\begin{equation*} \int_{K_{p,E}} |a_{1,1}(T)|^2 \, dT = \int_{K_{p,E}} |a_{1,1}(UT)|^2 \, dT \quad \hbox{and}\quad
\int_{K_{p,E}} a_{1,1}(T)^2 \, dT = \int_{K_{p,E}} a_{1,1}(UT)^2 \, dT, \end{equation*}
we must have 
\begin{equation*} \int_{K_{p,E}} 2\cos(\theta)\sin(\theta)\Re\bigl(a_{1,1}(T)\overline{a_{2,1}(T)}\bigr) \, dT 
= \int_{K_{p,E}} 2\cos(\theta)\sin(\theta)\, a_{1,1}(T)a_{2,1}(T)\, dT= 0. \end{equation*}
These in the real case are the same thing and show that the first integral in \eqref{eqp2:lem-permute-matrix} is 0.
In the complex and quaternion cases, we should also use as $A$ linear combinations of permutation matrices
with coefficients from $\{1,i,j,k\}\cap {\mathbb F}$ to deduce first that
\begin{equation}\label{eqp4:lem-permute-matrix} 
\begin{split}
 \int_{K_{p,E}} \Re(a_{1,1})\Re(a_{2,1}) \, dT = \int_{K_{p,E}} \Im_1(a_{1,1})\Im_1(a_{2,1}) \, dT = \cdots , 
\\
 \int_{K_{p,E}} \Re(a_{1,1})\Im_1(a_{2,1}) \, dT = \int_{K_{p,E}} \Im_1(a_{1,1})\Re(a_{2,1}) \, dT,
 \\
 \int_{K_{p,E}} 2\Re(a_{1,1})^2\, dT =  \int_{K_{p,E}} 2\Im_1(a_{1,1})^2\, dT =  \int_{K_{p,E}} \Re\bigl((1+i)a_{1,1}\bigr)^2\, dT = \cdots,
\end{split}
\end{equation}
and so on.

Finally, to also show that the second integral in \eqref{eqp2:lem-permute-matrix} is 0, note that 
\begin{align*} 
0 & = \int_{K_{p,E}} a_{1,1}(T)a_{2,1}(T) \, dT  =\int_{K_{p,E}} a_{1,1}(T)a_{1,2}(T) \, dT 
\\
& = \int_{K_{p,E}} a_{1,1}(UT)a_{1,2}(UT) \, dT 
\\
& = \int_{K_{p,E}} \bigl(\cos(\theta)a_{1,1}(T) + \sin(\theta)a_{2,1}(T)\bigr)\bigl(\cos(\theta)a_{1,2}(T) + \sin(\theta)a_{2,2}(T)\bigr) \, dT
\\
& = \int_{K_{p,E}} \cos^2(\theta) a_{1,1}(T) a_{1,2}(T) \, dT + \int_{K_{p,E}} \sin^2(\theta) a_{2,1}(T) a_{2,2}(T) \, dT
\\
&\hspace{1.4cm} + \cos(\theta) \sin(\theta) \left[\int_{K_{p,E}} a_{1,1}(T)a_{2,2}(T) \, dT + \int_{K_{p,E}} a_{2,1}(T)a_{1,2}(T) \, dT  \right]
\\
& = 2\cos(\theta) \sin(\theta) \int_{K_{p,E}} a_{1,1}(T)a_{2,2}(T) \, dT.
\end{align*}
This shows that $\displaystyle \int_{K_{p,E}} a_{1,1}(T)a_{2,2}(T) \, dT = 0$ and completes the proof
(again, in the complex and quaternion cases, if we combine it with equalities from \eqref{eqp4:lem-permute-matrix}).
\end{proof}

The next proposition is about how the integrals appearing in substatement 3 of Lemma \ref{lem:permute-matrix} relate to each other.

\begin{proposition}\label{prop:permute-matrix}
Suppose $p\gs 1$ and $E={\cal M}_n({\mathbb F})$ with ${\cal M}_n({\mathbb F}) = {\cal M}_n({\mathbb R})$ 
or ${\cal M}_n({\mathbb C})$ or ${\cal M}_n({\mathbb H})$.
Then 
\begin{equation*} \int_{K_{p,E}} |a_{1,1}|^2|a_{1,2}|^2\,dT = \int_{K_{p,E}} |a_{1,1}|^2|a_{2,2}|^2\,dT + 
\frac{2}{\beta}\int_{K_{p,E}} a_{1,1}\overline{a_{2,1}}a_{2,2}\overline{a_{1,2}}\,dT,\end{equation*}
where $\beta =1$ if ${\mathbb F} = {\mathbb R}$, $\beta = 2$ if ${\mathbb F} = {\mathbb C}$, and $\beta =4$ if ${\mathbb F} = {\mathbb H}$.
\end{proposition}
\begin{proof}
We apply \eqref{eqp1:lem-permute-matrix} again with $A$ being multiplication from the left by the rotation matrix $U$ in
\eqref{eqp3:lem-permute-matrix}: we obtain
\begin{gather*} 
 \int_{K_{p,E}} |a_{1,1}(T)|^2\cdot |a_{2,2}(T)|^2\,dT  =  \int_{K_{p,E}} |a_{1,1}(UT)|^2\cdot |a_{2,2}(UT)|^2\,dT
\\
 =  \int_{K_{p,E}} \big|\cos(\theta)a_{1,1}(T) + \sin(\theta)a_{2,1}(T)\big|^2\cdot \big|-\sin(\theta)a_{1,2}(T) + \cos(\theta)a_{2,2}(T)\big|^2 \, dT
 \\
 = \int_{K_{p,E}} \cos^2(\theta)\sin^2(\theta)\bigl(|a_{1,1}(T)|^2|a_{1,2}(T)|^2 + |a_{2,1}(T)|^2|a_{2,2}(T)|^2\bigr)\,dT 
 \\
 \quad + \int_{K_{p,E}} \bigl(\cos^4(\theta)|a_{1,1}(T)|^2|a_{2,2}(T)|^2 + \sin^4(\theta)|a_{2,1}(T)|^2|a_{1,2}(T)|^2\bigr)\,dT 
 \\
 \quad + \int_{K_{p,E}}\cos(\theta)\sin^3(\theta)\Bigl(2\Re\bigl(a_{1,1}(T) \overline{a_{2,1}(T)}\bigr)\cdot |a_{1,2}(T)|^2 - 
 2\Re\bigl(a_{2,2}(T) \overline{a_{1,2}(T)} \bigr)\cdot  |a_{2,1}(T)|^2\Bigr)\,dT 
 \\
 \quad + \int_{K_{p,E}}\cos^3(\theta)\sin(\theta)\Bigl(2\Re\bigl(a_{1,1}(T) \overline{a_{2,1}(T)}\bigr)\cdot |a_{2,2}(T)|^2 - 
  2\Re\bigl(a_{2,2}(T) \overline{a_{1,2}(T)} \bigr)\cdot |a_{1,1}(T)|^2\Bigr)\,dT + 
 \\
 \quad  - 
 \int_{K_{p,E}}\cos^2(\theta)\sin^2(\theta)\Bigl(2\Re\bigl(a_{1,1}(T) \overline{a_{2,1}(T)}\bigr)\cdot 2\Re\bigl(a_{2,2}(T) \overline{a_{1,2}(T)}\bigr)\Bigr)\,dT. 
\end{gather*}
Given that
\begin{align*}
\int_{K_{p,E}}\Re\bigl(a_{1,1}(T) \overline{a_{2,1}(T)}\bigr)\cdot |a_{1,2}(T)|^2\,dT &= 
\int_{K_{p,E}} \Re\bigl(a_{2,2}(T) \overline{a_{1,2}(T)} \bigr)\cdot |a_{2,1}(T)|^2\, dT
\\
= \int_{K_{p,E}}\Re\bigl(a_{1,1}(T) \overline{a_{2,1}(T)}\bigr)\cdot |a_{2,2}(T)|^2\,dT &= 
\int_{K_{p,E}} \Re\bigl(a_{2,2}(T) \overline{a_{1,2}(T)} \bigr)\cdot |a_{1,1}(T)|^2\, dT
\end{align*}
(which follows if we use \eqref{eqp1:lem-permute-matrix} with $A$ given by suitable permutation matrices),
we immediately see that
\begin{equation*} \int_{K_{p,E}} |a_{1,1}|^2|a_{1,2}|^2\,dT = \int_{K_{p,E}} |a_{1,1}|^2|a_{2,2}|^2\,dT + 
2\int_{K_{p,E}} \Re\bigl(a_{1,1}(T) \overline{a_{2,1}(T)}\bigr)\cdot \Re\bigl(a_{2,2}(T) \overline{a_{1,2}(T)}\bigr)\,dT.\end{equation*}
Now, combining Lemmas \ref{lem:sing-val-and-entries} and \ref{lem:permute-matrix}, we have that
\begin{equation*} \int_{K_{p,E}} a_{1,1}\overline{a_{2,1}}a_{2,2}\overline{a_{1,2}}\,dT\end{equation*}
is a real number, and hence
\begin{align*} 
\int_{K_{p,E}} a_{1,1}\overline{a_{2,1}}a_{2,2}\overline{a_{1,2}}\,dT &= 
\int_{K_{p,E}} \Re\bigl(a_{1,1}\overline{a_{2,1}}a_{2,2}\overline{a_{1,2}}\bigr)\,dT
\\
 &= \int_{K_{p,E}} a_{1,1}a_{2,1}a_{2,2}a_{1,2}\,dT
\end{align*}
if $E = {\cal M}_n({\mathbb R})$, or
\begin{gather*}
 = \int_{K_{p,E}} \Re\bigl(a_{1,1}(T) \overline{a_{2,1}(T)}\bigr)\cdot \Re\bigl(a_{2,2}(T) \overline{a_{1,2}(T)}\bigr)\,dT
- \int_{K_{p,E}} \Im\bigl(a_{1,1}(T) \overline{a_{2,1}(T)}\bigr)\cdot \Im\bigl(a_{2,2}(T) \overline{a_{1,2}(T)}\bigr)\,dT
\\
\intertext{if $E = {\cal M}_n({\mathbb C})$, or}
 = \int_{K_{p,E}} \Re\bigl(a_{1,1}(T) \overline{a_{2,1}(T)}\bigr)\cdot \Re\bigl(a_{2,2}(T) \overline{a_{1,2}(T)}\bigr)\,dT
- \int_{K_{p,E}} \Im_1\bigl(a_{1,1}(T) \overline{a_{2,1}(T)}\bigr)\cdot \Im_1\bigl(a_{2,2}(T) \overline{a_{1,2}(T)}\bigr)\,dT
\\
\quad  -  \int_{K_{p,E}} \Im_2\bigl(a_{1,1}(T) \overline{a_{2,1}(T)}\bigr)\cdot \Im_2\bigl(a_{2,2}(T) \overline{a_{1,2}(T)}\bigr)\,dT
-  \int_{K_{p,E}} \Im_3\bigl(a_{1,1}(T) \overline{a_{2,1}(T)}\bigr)\cdot \Im_3\bigl(a_{2,2}(T) \overline{a_{1,2}(T)}\bigr)\,dT
\end{gather*}
if $E = {\cal M}_n({\mathbb H})$. Applying \eqref{eqp1:lem-permute-matrix} with suitable permutation matrices again 
(or linear combinations of such matrices with coefficients from $\{1,i,j,k\}\cap {\mathbb F}$), we conclude that
\begin{equation*}
\int_{K_{p,E}} a_{1,1}\overline{a_{2,1}}a_{2,2}\overline{a_{1,2}}\,dT = 
\beta \int_{K_{p,E}} \Re\bigl(a_{1,1}(T) \overline{a_{2,1}(T)}\bigr)\cdot \Re\bigl(a_{2,2}(T) \overline{a_{1,2}(T)}\bigr)\,dT.
\end{equation*}
The conclusion of the lemma follows.
\end{proof}

\begin{corollary}\label{cor:order-of-magnitudes}
We have that
\begin{equation}\label{eq1:cor-order-of-magnitudes}
 \Big|\int_{\overline{K}_{p,{\cal M}_n({\mathbb F})}} a_{1,1}\overline{a_{2,1}}a_{2,2}\overline{a_{1,2}}\,dT\,\Big| \lesssim \frac{1}{n} 
\left(\int_{\overline{K}_{p,{\cal M}_n({\mathbb F})}} |a_{1,1}|^2\,dT\right)^2 = \frac{1}{n}\,\beta^2 L_{K_{p,{\cal M}_n({\mathbb F})}}^4, 
\end{equation}
and
\begin{multline} \label{eq2:cor-order-of-magnitudes}
\int_{\overline{K}_{p,{\cal M}_n({\mathbb F})}} |a_{1,1}|^2|a_{1,2}|^2\,dT, \ \  
\int_{\overline{K}_{p,{\cal M}_n({\mathbb F})}} |a_{1,1}|^2|a_{2,2}|^2\,dT
\\
 = (1 + o(1)) \left(\int_{\overline{K}_{p,{\cal M}_n({\mathbb F})}} |a_{1,1}|^2\,dT\right)^2 
 = (1 + o(1))\,\beta^2 L_{K_{p,{\cal M}_n({\mathbb F})}}^4,
\end{multline}
where $\beta\in \{1,2,4\}$ is as above.
\end{corollary}
\begin{proof}
In Proposition \ref{prop:order-of-magnitudes} we saw that
\begin{equation*}\frac{M_p\bigl(\|x\|_4^4\bigr)}{M_p(1)} = n\frac{M_p\bigl(x_1^4\bigr)}{M_p(1)} \simeq 
n \left(\frac{M_p\bigl(x_1^2\bigr)}{M_p(1)}\right)^2 \simeq n\cdot n^{4/p}. \end{equation*}
By reversing the identities that Lemma \ref{lem:reduction to n-integrals1} gives us, we can write
\begin{align*} \frac{1}{|K_{p,{\cal M}_n({\mathbb F})}|}\,\int_{K_{p,{\cal M}_n({\mathbb F})}}|a_{1,1}|^2\,dT
&= \frac{1}{n^2} \frac{1}{|K_{p,{\cal M}_n({\mathbb F})}|}\,\int_{K_{p,{\cal M}_n({\mathbb F})}}\|s(T)\|_2^2\,dT
\\ 
&\simeq \frac{d_n^{-2/p}}{n^2} \frac{M_p\bigl(\|x\|_2^2\bigr)}{M_p(1)} = \frac{d_n^{-2/p}}{n} \frac{M_p\bigl(x_1^2\bigr)}{M_p(1)} \simeq n^{-1-2/p},
\end{align*}
as well as 
\begin{equation*}\frac{1}{|K_{p,{\cal M}_n({\mathbb F})}|}\,\int_{K_{p,{\cal M}_n({\mathbb F})}}\|s(T)\|_4^4\,dT
\simeq d_n^{-4/p} \frac{M_p\bigl(\|x\|_4^4\bigr)}{M_p(1)} \simeq \frac{n}{n^{4/p}}.\end{equation*}
This 
implies that
\begin{equation*} \int_{\overline{K}_{p,{\cal M}_n({\mathbb F})}} \|s(T)\|_4^4\,dT \simeq n^3\cdot 
\left(\int_{\overline{K}_{p,{\cal M}_n({\mathbb F})}} |a_{1,1}|^2\,dT\right)^2.\end{equation*}
But by Lemmas \ref{lem:sing-val-and-entries} and \ref{lem:permute-matrix}, we know that
\begin{multline*}
\int_{\overline{K}_{p,{\cal M}_n({\mathbb F})}} \|s(T)\|_4^4\,dT 
= n^2\cdot \int_{\overline{K}_{p,{\cal M}_n({\mathbb F})}} |a_{1,1}|^4\, dT + 
2n^2(n-1)\cdot \int_{\overline{K}_{p,{\cal M}_n({\mathbb F})}}|a_{1,1}|^2|a_{1,2}|^2\,dT 
\\
+ n^2(n-1)^2\cdot \int_{\overline{K}_{p,{\cal M}_n({\mathbb F})}} a_{1,1}\overline{a_{2,1}}a_{2,2}\overline{a_{1,2}}
\gs n^2(n-1)^2\cdot \int_{\overline{K}_{p,{\cal M}_n({\mathbb F})}} a_{1,1}\overline{a_{2,1}}a_{2,2}\overline{a_{1,2}}.
\end{multline*}
Moreover, since $\int_{\overline{K}_{p,{\cal M}_n({\mathbb F})}} \|s(T)\|_4^4\,dT > 0$, we also have that
\begin{align*} -n^2(n-1)^2\cdot \int_{\overline{K}_{p,{\cal M}_n({\mathbb F})}} a_{1,1}\overline{a_{2,1}}a_{2,2}\overline{a_{1,2}}
& <  n^2\cdot \int_{\overline{K}_{p,{\cal M}_n({\mathbb F})}} |a_{1,1}|^4\, dT + 
2n^2(n-1)\cdot \int_{\overline{K}_{p,{\cal M}_n({\mathbb F})}}|a_{1,1}|^2|a_{1,2}|^2\,dT
\\
& \ls C n^2(2n-1) \cdot \left(\int_{\overline{K}_{p,{\cal M}_n({\mathbb F})}} |a_{1,1}|^2\,dT\right)^2, \end{align*}
where the last inequality is a consequence of the Cauchy-Schwarz inequality and 
of standard properties of convex bodies (see e.g. \cite[Theorem 2.4.6]{BGVV-book}).
Inequality \eqref{eq1:cor-order-of-magnitudes} follows.

To also establish \eqref{eq2:cor-order-of-magnitudes}, we recall that
\begin{align*}
\sigma_{K_{p,{\cal M}_n({\mathbb F})}}^2\cdot d_n &\simeq {\rm Var}_{\overline{K}_{p,E}}\bigl(\|T\|_{HS}^2\bigr) 
\\
& = n^2\cdot \int_{\overline{K}_{p,{\cal M}_n({\mathbb F})}} |a_{1,1}|^4\, dT \ +\ 
2n^2(n-1)\cdot \int_{\overline{K}_{p,{\cal M}_n({\mathbb F})}}|a_{1,1}|^2|a_{1,2}|^2\,dT 
\\
&\ \quad + n^2(n-1)^2\cdot  \int_{\overline{K}_{p,{\cal M}_n({\mathbb F})}}|a_{1,1}|^2|a_{2,2}|^2\,dT \ -\  
n^4 \cdot \left(\int_{\overline{K}_{p,{\cal M}_n({\mathbb F})}} |a_{1,1}|^2\,dT\right)^2.
\end{align*}
Since by \cite{Barthe-Cordero-2013} we know that $\sigma_{K_{p,{\cal M}_n({\mathbb F})}}^2 = O(n)$
for all $p\gs 1$, we can infer that
\begin{equation*} \Bigg|\int_{\overline{K}_{p,{\cal M}_n({\mathbb F})}}|a_{1,1}|^2|a_{2,2}|^2\,dT - 
\left(\int_{\overline{K}_{p,{\cal M}_n({\mathbb F})}} |a_{1,1}|^2\,dT\right)^2\Bigg|  = 
O\left(\frac{1}{n}\right) \left(\int_{\overline{K}_{p,{\cal M}_n({\mathbb F})}} |a_{1,1}|^2\,dT\right)^2. \end{equation*}
Furthermore, combining this with Proposition \ref{prop:permute-matrix} and \eqref{eq1:cor-order-of-magnitudes},
we get the same conclusion for the difference $\displaystyle \int_{\overline{K}_{p,{\cal M}_n({\mathbb F})}}|a_{1,1}|^2|a_{1,2}|^2\,dT - 
\left(\int_{\overline{K}_{p,{\cal M}_n({\mathbb F})}} |a_{1,1}|^2\,dT\right)^2$, as claimed. 
\end{proof}

\medskip

We are finally in a position to establish a type of negative correlation property for the original, uniform measures on 
$K_{p,{\cal M}_n({\mathbb F})}$ as well: this can be done for $p$ for which the estimate in \eqref{eq1:thm:neg-cor-prop}
is accurate, or close to it.

\begin{theorem} \label{thm:neg-cor-prop1}
Let $p$ be such that
\begin{equation}\label{eq1:thm:neg-cor-prop1} \frac{M_{2,\beta,\beta-1,p}\bigl(x_1^4\bigr)}{M_{2,\beta,\beta-1,p}(1)} 
< \left(2 + o(1)\right) \left(\frac{M_{2,\beta,\beta-1,p}\bigl(x_1^2\bigr)}{M_{2,\beta,\beta-1,p}(1)}\right)^2, \end{equation}
where $\beta =1$ if ${\mathbb F} = {\mathbb R}$, $\beta = 2$ if ${\mathbb F} = {\mathbb C}$, and $\beta =4$ if ${\mathbb F} = {\mathbb H}$,
and suppose in addition that $K_{p,{\cal M}_n({\mathbb F})}$ satisfies the thin-shell conjecture, or at least 
that $\sigma_{K_{p,{\cal M}_n({\mathbb F})}}^2  = o(n)$.
Then for every $i, j, k\in\{1,\ldots, n\}$, $j\neq k$, we have
\begin{align*}
 \int_{\overline{K}_{p,{\cal M}_n({\mathbb F})}} |a_{i,j}|^2|a_{i,k}|^2\,dT 
 &= \int_{\overline{K}_{p,{\cal M}_n({\mathbb F})}} |a_{j,i}|^2|a_{k,i}|^2\,dT 
 \\
 & < \left(\int_{\overline{K}_{p,{\cal M}_n({\mathbb F})}} |a_{i,j}|^2\,dT\right) \left(\int_{\overline{K}_{p,{\cal M}_n({\mathbb F})}} |a_{i,k}|^2\,dT\right)
 = \beta^2 L_{K_{p,{\cal M}_n({\mathbb F})}}^4.
\end{align*}
\end{theorem}
\begin{proof}
Let us write 
\begin{equation} \label{eqp1:thm:neg-cor-prop1}
 \frac{M_{2,\beta,\beta-1,p}\bigl(x_1^4\bigr)}{M_{2,\beta,\beta-1,p}(1)} 
 = c_4\, \left(\frac{M_{2,\beta,\beta-1,p}\bigl(x_1^2\bigr)}{M_{2,\beta,\beta-1,p}(1)}\right)^2
 \end{equation}
where, by Section \ref{sec:main-proofs} and the assumption of the theorem, we know that $\dfrac{3}{2} + o(1) \ls c_4 < 2 + o(1)$.
We start by recalling that
\begin{align*} 
\frac{M_p\bigl(x_1^4\bigr)}{M_p(1)} &= 
\frac{1}{n}\cdot \frac{\Gamma\left(1+\frac{d_n+4}{p}\right)}{\Gamma\left(1+\frac{d_n}{p}\right)}
\frac{1}{|K_{p,{\cal M}_n({\mathbb F})}|}\int_{K_{p,{\cal M}_n({\mathbb F})}}\|s(T)\|_4^4\,dT,
\\
\intertext{and that}
\left(\frac{M_p\bigl(x_1^2\bigr)}{M_p(1)}\right)^2& = \frac{1}{n^2} \cdot 
\left(\frac{\Gamma\left(1+\frac{d_n+2}{p}\right)}{\Gamma\left(1+\frac{d_n}{p}\right)}\right)^2
\left(\frac{1}{|K_{p,{\cal M}_n({\mathbb F})}|}\int_{K_{p,{\cal M}_n({\mathbb F})}}\|s(T)\|_2^2\,dT\right)^2
\\
& = n^2 \cdot 
\left(\frac{\Gamma\left(1+\frac{d_n+2}{p}\right)}{\Gamma\left(1+\frac{d_n}{p}\right)}\right)^2
\left(\frac{1}{|K_{p,{\cal M}_n({\mathbb F})}|}\int_{K_{p,{\cal M}_n({\mathbb F})}}|a_{1,1}|^2\,dT\right)^2.
\end{align*}
Similarly,
\begin{equation}\label{eqp2:thm:neg-cor-prop1}  \frac{M_p\bigl(x_1^2x_2^2\bigr)}{M_p(1)} 
= \frac{1}{n(n-1)}\cdot \frac{\Gamma\left(1+\frac{d_n+4}{p}\right)}{\Gamma\left(1+\frac{d_n}{p}\right)}
\frac{1}{|K_{p,{\cal M}_n({\mathbb F})}|}\int_{K_{p,{\cal M}_n({\mathbb F})}}\sum_{i\neq j}s_i^2s_j^2\,dT.
 \end{equation}
We now combine these identities with  
Proposition \ref{prop:permute-matrix},
identities \eqref{eq1:lem:sing-val-and-entries}-\eqref{eq2:lem:sing-val-and-entries}, and the assumptions that
$c_4 < 2 + o(1)$ and $\sigma_{K_{p,{\cal M}_n({\mathbb F})}}^2  = o(n)$: we first see that, because of \eqref{eqp1:thm:neg-cor-prop1},
we must have
\begin{multline*} 
n\cdot \int_{\overline{K}_p} |a_{1,1}|^4\, dT \ + \
2n(n-1)\cdot \int_{\overline{K}_p}|a_{1,1}|^2|a_{1,2}|^2\,dT 
\ +\  n(n-1)^2\cdot \int_{\overline{K}_p} a_{1,1}\overline{a_{2,1}}a_{2,2}\overline{a_{1,2}}
\\
= \left(c_4 + O\bigl(1/n^2\bigr)\right)\,n^2\cdot \left(\int_{\overline{K}_p}|a_{1,1}|^2\,dT\right)^2.
\end{multline*}
Given that $\displaystyle \int_{\overline{K}_p}|a_{1,1}|^2|a_{1,2}|^2\,dT = \bigl(1+O(1/n)\bigr) \left(\int_{\overline{K}_p}|a_{1,1}|^2\,dT\right)^2$,
it follows that
\begin{equation} \label{eqp3:thm:neg-cor-prop1}
\int_{\overline{K}_p} a_{1,1}\overline{a_{2,1}}a_{2,2}\overline{a_{1,2}}\, dT \ls
\left(\frac{c_4 -2}{n} + O\left(\frac{1}{n^2}\right)\right)\cdot \left(\int_{\overline{K}_p}|a_{1,1}|^2\,dT\right)^2. 
\end{equation}
But now recall that, because of Proposition \ref{prop:Var Kp and Mp}, 
the assumption $\sigma_{K_{p,{\cal M}_n({\mathbb F})}}^2  = o(n)$ implies that
\begin{equation*} \frac{M_p\bigl(x_1^2x_2^2\bigr)}{M_p(1)} \ls  
\left(1 - \frac{c_4-1}{n} + o\left(\frac{1}{n}\right)\right) \left(\frac{M_p\bigl(x_1^2\bigr)}{M_p(1)}\right)^2\end{equation*}
(where $o(1/n)$ is at least $O(1/n^2)$ here, but may be larger if $\overline{K}_{p,{\cal M}_n({\mathbb F})}$ does not
satisfy the thin-shell conjecture).
This, through Proposition \ref{prop:permute-matrix}, and equations \eqref{eq2:lem:sing-val-and-entries} and \eqref{eqp2:thm:neg-cor-prop1},
translates into
\begin{align*}
&n(n-1)\left(\int_{\overline{K}_p}|a_{1,1}|^2|a_{1,2}|^2\,dT - \bigl(1+ 2/\beta\bigr) \int_{\overline{K}_p} a_{1,1}\overline{a_{2,1}}a_{2,2}\overline{a_{1,2}}\right) 
\\
=\ & n(n-1) \left(\int_{\overline{K}_p}|a_{1,1}|^2|a_{2,2}|^2\,dT -  \int_{\overline{K}_p} a_{1,1}\overline{a_{2,1}}a_{2,2}\overline{a_{1,2}}\right)
\\
\ls\ & \left(1 + \frac{1-c_4}{n} + o\left(\frac{1}{n}\right)\right) n^2\cdot \left(\int_{\overline{K}_p}|a_{1,1}|^2\,dT\right)^2,
\end{align*}
which combined with \eqref{eqp3:thm:neg-cor-prop1} (and Lemma \ref{lem:permute-matrix}) implies the claim of the theorem.
\end{proof}

\smallskip

Here are some concluding remarks concerning this theorem:
\begin{itemize}
\item Note that this negative correlation property is again a necessary condition for the thin-shell conjecture to be true
for $p$ for which \eqref{eq1:thm:neg-cor-prop1} is true. These include all $p\gtrsim \log n$ (in fact, it is not difficult to see 
that $c_4$ can be as close to $3/2 + o(1)$ in these cases if the implied absolute constant in the latter inequality is sufficiently large). 
We should clarify however that we cannot expect \eqref{eq1:thm:neg-cor-prop1} to be true for all $p$: 
for example, for the Euclidean ball ($p=2$) we know that all cross terms are equal, that is
\begin{equation*} \int_{K_{2, {\cal M}_n({\mathbb F})}} |a_{1,1}|^2|a_{1,2}|^2\,dT=  
\int_{K_{2, {\cal M}_n({\mathbb F})}} |a_{1,1}|^2|a_{2,1}|^2\,dT = \int_{K_{2, {\cal M}_n({\mathbb F})}} |a_{1,1}|^2|a_{2,2}|^2\,dT;
\end{equation*}
then, by Proposition \ref{prop:permute-matrix}, we see that 
$\displaystyle \int_{K_{2, {\cal M}_n({\mathbb F})}} a_{1,1}\overline{a_{2,1}}a_{2,2}\overline{a_{1,2}} \, dT = 0$, and hence
\begin{equation*} \frac{M_{2,\beta,\beta-1,2}\bigl(x_1^4\bigr)}{M_{2,\beta,\beta-1,2}(1)} 
= \left(2 + o(1)\right) \left(\frac{M_{2,\beta,\beta-1,2}\bigl(x_1^2\bigr)}{M_{2,\beta,\beta-1,2}(1)}\right)^2.\end{equation*}
Recall that for the Euclidean ball we know that all cross terms are 
\begin{equation*} < \left(\int_{K_{2, {\cal M}_n({\mathbb F})}} |a_{1,1}|^2\,dT\right)^2,\end{equation*}
simply because $\sigma_{K_{2, {\cal M}_n({\mathbb F})}}^2 = O(1/n^2)$, and not $O(1)$.

Another case for which \eqref{eq1:thm:neg-cor-prop1} is not true is the case of $p = 1$: we have that
\begin{equation*} \frac{M_1\bigl(x_1^4\bigr)}{M_1(1)} 
\gs \left(\frac{17}{8} + o(1)\right) \left(\frac{M_1\bigl(x_1^2\bigr)}{M_1(1)}\right)^2,\end{equation*} 
which moreover implies that in this case it is the cross terms $\int_{\overline{K}_1}|a_{i,j}|^2|a_{l,k}|^2\,dT$ with $i\neq l$, $j\neq k$, 
which are the smallest ones.
\item  The assumption $\sigma_{K_{p, {\cal M}_n({\mathbb F})}}^2 = o(n)$ can be relaxed a little, and replaced by
the assumption $\sigma_{K_{p, {\cal M}_n({\mathbb F})}}^2 \ls c_0n$
(with a constant that may be smaller than the one guaranteed by \cite{Barthe-Cordero-2013} however): for example,
we can have the same conclusion to the theorem if we take $c_0$ to be sufficiently small and we also assume
\begin{equation}\label{alt-eq1:thm:neg-cor-prop1} 
\frac{M_{2,\beta,\beta-1,p}\bigl(x_1^4\bigr)}{M_{2,\beta,\beta-1,p}(1)} 
\ls \left(\frac{9}{5} + o(1)\right)
\left(\frac{M_{2,\beta,\beta-1,p}\bigl(x_1^2\bigr)}{M_{2,\beta,\beta-1,p}(1)}\right)^2 
\end{equation}
say. Since the latter estimate is satisfied anyway when $p\gs c_1 n\log n$, and since we also saw in Section \ref{sec:main-proofs}
that $\sigma_{K_{p, {\cal M}_n({\mathbb F})}}^2 \lesssim n$ for such $p$ (and the implied constant can be made as small as we want
as long as $c_1$ is sufficiently large), this gives us the range of $p$ for which we already know that the theorem can be applied, 
and that the stated negative correlation property holds true anyway.
\item As mentioned earlier, this negative correlation property is a necessary condition for the thin-shell conjecture to be true
for some of the balls $K_{p, {\cal M}_n({\mathbb F})}$, but does not appear to be a sufficient one too. In fact,
our arguments do not seem to allow us to distinguish between the cases
\begin{multline*}
\int_{\overline{K}_p}|a_{1,1}|^2|a_{2,2}|^2\,dT < \left(\int_{\overline{K}_p}|a_{1,1}|^2\,dT\right)^2
\\
\hbox{or}\qquad \int_{\overline{K}_p}|a_{1,1}|^2|a_{2,2}|^2\,dT = \left(1 + \frac{c}{n^2}\right) \left(\int_{\overline{K}_p}|a_{1,1}|^2\,dT\right)^2. 
\end{multline*} 
Nevertheless it still seems like a question of independent interest to study for which other indices $p$, if any,
we have some sort of negative correlation property as above, or even to try to re-establish the property for the known cases
in a more direct manner, that is, without having to go through estimates for $\sigma_{K_p}$ 
(if the latter turns out to be possible, it would immediately give us one more proof of the estimate $\sigma_{K_p}^2 = O(n)$ 
from \cite{Barthe-Cordero-2013} as well).   
\end{itemize}

\bigskip

\noindent{\bf Acknowledgements.} The authors wish to thank 
Roman Vershynin, Grigoris Paouris and Franck Barthe for very useful comments and references. 
The second-named author would also like to thank Mark Rudelson, Santosh Vempala and Matus Telgarsky
for inspiring and helpful conversations and for their encouragement. This work started
as part of a project from the REU programme of the University of Michigan of Summer 2015; 
during the programme, the first-named author was supported by the National Science Foundation 
under grant number DMS 1265782.

\bigskip

\bigskip

\small

\noindent \textsc{Jordan Radke:} Mathematics Department, Princeton University, 
Fine Hall, Washington Road, Princeton, NJ 08544-1000, USA

\noindent {\it E-mail:} \texttt{jradke@princeton.edu}

\medskip

\noindent \textsc{Beatrice-Helen Vritsiou:} Department of Mathematics,
University of Michigan, 2074 East Hall, 530 Church Street, Ann Arbor, MI  48109-1043, USA

\noindent {\it E-mail:} \texttt{vritsiou@umich.edu}

\medskip


\bigskip

\begin{thebibliography}{88}

\footnotesize

\bibitem{Anderson-Guionnet-Zeitouni-2010} Anderson, G. W., A. Guionnet and O. Zeitouni.
{\sl An Introduction to Random Matrices.} Cambridge Studies in Advanced Mathematics {\bf 118}. 
Cambridge: Cambridge University Press, 2010.

\bibitem{Anttila-Ball-Perissinaki-2003} Anttila, M., K. Ball and I. Perissinaki. ``The central limit problem for convex bodies.'' 
Transactions of the American Mathematical Society {\bf 355}, no. 12 (2003): 4723-4735.

\bibitem{Aubrun-Szarek-2006} Aubrun, G. and S. Szarek. 
``Tensor products of convex sets and the volume of separable states on $N$ qudits.'' 
Physical Review A {\bf 73}, no. 2 (2006): 022109.

\bibitem{Ball-Perissinaki-1998} Ball, K. and I. Perissinaki. ``The subindependence of coordinate slabs in $\ell_p^n$ balls.''
{\sl Israel Journal of Mathematics} {\bf 107} (1998): 289-299.


\bibitem{Barthe-Cordero-2013} Barthe, F. and D. Cordero-Erausquin. ``Invariances in variance estimates.''
{\sl Proceedings of the London Mathematical Society} {\bf 106}, no. 3 (2013): 33-64.

\bibitem{Bobkov-Koldobsky-2003}
Bobkov, S. and A. Koldobsky. ``On the central limit property of convex bodies.''
{\sl Geometric aspects of functional analysis, Lecture Notes in Mathematics} {\bf 1807} (2003): 44-52. 
Berlin: Springer, 2003.

\bibitem{Bourgain-1991} Bourgain, J. ``On the distribution of polynomials on high dimensional convex sets.'' 
{\sl Geometric aspects of functional analysis, Lecture Notes in Mathematics} {\bf 1469} (1989-90): 127-137.
Berlin: Springer, 1991.


\bibitem{BGVV-book}
Brazitikos, S., A. Giannopoulos, P. Valettas and B.-H. Vritsiou.
{\sl Geometry of Isotropic Convex Bodies.}
Mathematical Surveys and Monographs {\bf 196}. 
American Mathematical Society, Providence, RI, 2014.

 \bibitem{Diaconis-Freedman-1984}
 Diaconis, P. and D. Freedman. 
 ``Asymptotics of graphical projection pursuit.'' 
 {\sl The Annals of Statistics} {\bf 12}, no. 3 (1984): 793-815.

\bibitem{Edelman-LaCroix-2015}
Edelman, A. and M. La Croix. 
``The singular values of the GUE (less is more).'' 
{\sl Random Matrices. Theory and Applications} {\bf 4}, no. 4 (2015): 1550021.

\bibitem{Eldan-2013}
Eldan, R. ``Thin shell implies spectral gap up to polylog via a stochastic localization scheme.'' 
{\sl Geometric and Functional Analysis} {\bf 23}, no. 2 (2013): 532-569.

\bibitem{Eldan-Klartag-2011}
Eldan, R. and B. Klartag.
``Approximately Gaussian marginals and the hyperplane conjecture.'' 
{\sl Concentration, functional inequalities and isoperimetry, Contemporary Mathematics} {\bf 545}: 55-68. 
American Mathematical Society, Providence, RI, 2011.



\bibitem{Fleury-2010b}
Fleury, B. ``Concentration in a thin Euclidean shell for log-concave measures.'' 
Journal of Functional Analysis {\bf 259}, no. 4 (2010): 832-841.

\bibitem{Fleury-Guedon-Paouris-2007} Fleury, B., O. Guédon and G. Paouris.
``A stability result for mean width of $L_p$-centroid bodies.'' Advances in Mathematics {\bf 214}, no. 2 (2007): 865-877.

\bibitem{Guedon-Milman-2011}
Guédon, O. and E. Milman.
``Interpolating thin-shell and sharp large-deviation estimates for isotropic log-concave measures.''
{\sl Geometric and Functional Analysis} {\bf 21}, no. 5 (2011): 1043-1068.

\bibitem{Guedon-Paouris-2007}
Guédon, O. and G. Paouris. ``Concentration of mass on the Schatten classes.''
{\sl Annales de l'Institut Henri Poincaré, Probabilité et Statistiques} {\bf 43}, no. 1 (2007): 87-99.

\bibitem{Hua-1963} Hua, L. K.
{\sl Harmonic Analysis of Functions of Several Complex Variables in the Classical Domains}
(translated from Russian by L. Ebner and A. Korányi).
American Mathematical Society, Providence, RI, 1963.


\bibitem{Klartag-2006} Klartag, B. ``On convex perturbations with a bounded isotropic
constant.'' {Geometric and Functional Analysis} {\bf 16}, no. 6 (2006): 1274-1290.

\bibitem{Klartag-2007a} Klartag, B.
``A central limit theorem for convex sets.'' 
{\sl Inventiones Mathematicae} {\bf 168}, no. 1 (2007): 91-131.

\bibitem{Klartag-2007b} Klartag, B.
``Power-law estimates for the central limit theorem for convex sets.'' 
{\sl Journal of Functional Analysis} {\bf 245}, no. 1 (2007): 284-310.

\bibitem{Klartag-2009} Klartag, B. 
``A Berry-Esseen type inequality for convex bodies with an unconditional basis.'' 
Probability Theory and Related Fields {\bf 145}, no. 1-2 (2009): 1-33.

\bibitem{Klartag-2013} Klartag, B. 
``Poincaré inequalities and moment maps.'' 
{\sl Annales de la Faculté des Sciences de Toulouse, Mathématiques, Série 6}\,
{\bf 22}, no. 1 (2013): 1-41. 


\bibitem{Klartag-EMilman-2012} Klartag, B. and E. Milman. ``Centroid bodies and the
logarithmic Laplace transform--a unified approach.'' {\sl Journal of Functional
Analysis} {\bf 262}, no. 1 (2012): 10-34.

\bibitem{Konig-Meyer-Pajor-1998} 
König, H., M. Meyer and A. Pajor. ``The isotropy constants of the Schatten classes are bounded.''
{\sl Mathematische Annalen} {\bf 312}, no. 4 (1998): 773-783.

\bibitem{Mehta-2004} Mehta, M. L. {\sl Random matrices.}
3rd edition, Pure and Applied Mathematics (Amsterdam) {\bf 142}. 
Amsterdam: Elsevier/Academic Press, 2004. 


\bibitem{Naor-Romik-2003}
Naor, A. and D. Romik. ``Projecting the surface measure of the sphere of $\ell_p^n$.''
{\sl Annales de l'Institut Henri Poincaré, Probabilité et Statistiques} {\bf 39} no. 2 (2003): 241-261.

\bibitem{Paouris-2006} Paouris, G. ``Concentration of mass in convex bodies.'' {\sl Geometric and Functional Analysis} {\bf 16}, no. 5 (2006): 1021-1049.

\bibitem{Pilipczuk-Wojtaszczyk-2008} Pilipczuk M. and J. O. Wojtaszczyk. 
``The negative association property for the absolute values of random variables equidistributed on a generalized
Orlicz ball.'' {\sl Positivity} {\bf 12}, no. 3 (2008): 421-474.

\bibitem{SaintRaymond-1984} Saint-Raymond, J. ``Le volume des idéaux d'opérateurs classiques.''
{\sl Studia Mathematica} {\bf 80}, no. 1 (1984): 63-75.

\bibitem{Sudakov-1978} Sudakov, V. N.
``Typical distributions of linear functionals in finite-dimensional spaces of high dimension.'' 
{\sl Doklady Akademii Nauk SSSR} {\bf 243}, no. 6 (1978): 1402-1405. 
(In Russian, English translation in Soviet Mathematics. Doklady {\bf 19} (1978): 1578-1582.)



\end{thebibliography}
\end{document}